\documentclass[onecolumn,11pt]{IEEEtran}

\usepackage[usenames,dvipsnames]{color}

\usepackage{amsmath,amsthm,amssymb,latexsym,cite}
\usepackage{graphicx}
\usepackage{subfigure}

\usepackage{url}

\newtheorem{corollary}{Corollary}
\newtheorem{lemma}{Lemma}
\newtheorem{theorem}{Theorem}

\newtheorem{assume}{Condition}

\def\expe{\mathbb{E}}   %
\def\prob{\mathbb{P}}   %

\newcommand{\norm}[1]{\left\| #1 \right\|}
\def\dmax{d_{\mathrm{max}}}
\newcommand{\trans}{{\top}}

\newcommand\defeq{\stackrel{\triangle}{=}}
\newcommand{\resp}[1]{#1}

\def\mc{\mathcal}
\def\mbf{\mathbf}

\begin{document}

\title{Distributed Learning of Distributions via Social Sampling}

\author{Anand D.~Sarwate,~\IEEEmembership{Member,~IEEE,} 
	and~Tara Javidi,~\IEEEmembership{Senior Member,~IEEE,}
        \thanks{
        Manuscript received May 17, 2013; revised January 9, 2014; second revision .  
        Date of current version \today.
The work of both authors was funded by National Science Foundation (NSF) Grant Number CCF-1218331.  A.D. Sarwate received additional support from the California Institute for Telecommunications and Information Technology (CALIT2) at UC San Diego.  T. Javidi received additional support from NSF Grant Numbers CCF-1018722 and CNS-1329819.  Simulations were supported by the UCSD FWGrid Project, NSF Research Infrastructure Grant Number EIA-0303622.
Parts of this work were presented at the 46th Annual Conference on Information Sciences and Systems (CISS), Princeton, NJ, USA, March 2012~\cite{SarwateJ:12ciss}, and the 49th Annual Allerton Conference on Communication, Control and Computation, Monticello, IL, USA, September 2011~\cite{SarwateJ:11allerton}. 
	}
	\thanks{Anand D. Sarwate is with the Department of Electrical and Computer Engineering, Rutgers, The State University of New Jersey, 84 Brett Road, Piscataway, NJ 08854, USA.  Email: \texttt{asarwate@ece.rutgers.edu}.
Tara Javidi is with the Department of Electrical and Computer Engineering, University of California, San Diego, 9500 Gilman Dr. MC 0407, La Jolla CA 92093-0407, USA. Email: \texttt{tjavidi@ucsd.edu}.
}%
\thanks{Copyright (c) 2014 IEEE. Personal use of this material is permitted. However, permission to use this material for any other purposes must be obtained from the IEEE by sending a request to \texttt{pubs-permissions@ieee.org}.}        
}

\maketitle

\begin{abstract}
A protocol for distributed estimation of discrete distributions is proposed.  Each agent begins with a single sample from the distribution, and the goal is to learn the empirical distribution of the samples.  The protocol is based on a simple message-passing model motivated by communication in social networks.  Agents sample a message randomly from their current estimates of the distribution, resulting in a protocol with quantized messages.  Using tools from stochastic approximation, the algorithm is shown to converge almost surely.  Examples illustrate three regimes with different consensus phenomena.  Simulations demonstrate this convergence and give some insight into the effect of network topology.
\end{abstract}

\section{Introduction}

The emergence of large-network paradigms for communications and the widespread adoption of social networking technologies has resurrected interest in classical models of opinion formation and distributed computation as well as new approaches to distributed learning and inference.  In this paper we propose a  simple message passing protocol inspired by social communication and show how it allows a network of individuals can learn about global phenomena.  In particular, we study a situation wherein each node or agent in a network holds an initial opinion and the agents communicate with each other to infer the distribution of their initial opinions.  Our model of messaging is a simple abstraction of social communication in which individuals exchange single opinions.  This model is a randomized approximation of consensus procedures.  Because agents collect samples of the opinions of their neighbors, we call our model \textit{social sampling}.

In our protocol agents merge the sampled opinions of their neighbors with their own estimates using a weighted average.  Averaging has been used to model opinion formation for decades, starting with the early work of French \cite{French:56power}, Harary \cite{Harary:59criterion}, and DeGroot \cite{DeGroot:74consensus}.  These works focused on averaging as a means to an end -- averaging the local opinions of a group of peers was a simple way to model the process of negotiation and compromise of opinions represented as scalar variables.  A natural extension of the above work is that where all agents are interested in the 
local reconstruction of the empirical distribution of discrete opinions.  Such locally constructed empirical distributions not only provide richer information about global network properties (such as the outcome of a vote, the confidence interval around the mean, etc), but from a statistical estimation perspective provide estimates of local sufficient statistics when the agents' opinions are independent and identically distributed (i.i.d.) samples from a common distribution.  %

For opinions taking value in a finite, discrete set, we can \textit{compute} the empirical distribution of opinions across a network by running an average consensus algorithm for each possible value of the opinion.  This can even be done in parallel so that at each time agents exchange their entire histogram of opinions and compute weighted averages of their neighbors' histograms to update their estimate.  In a social network, this would correspond to modeling the interaction of two agents as a complete exchange of their entire beliefs in every opinion, which is not realistic.  In particular, if the number of possible opinions is large (corresponding to a large number of bins or elements in the histogram), communicating information about all opinions may be very inefficient, especially if the true distribution of opinions is far from uniform.

In contrast, this paper considers a novel model in which agents' information is disseminated through randomly selected samples of locally constructed histograms~\cite{SarwateJ:11allerton}.  The use of random  samples results in a much lower overhead because it accounts for the popularity/frequency of histogram bins and naturally enables finite-precision communication among neighboring nodes.  It is not hard to guarantee that the expectation of the node estimates converges to the true histogram when the mean of any given (randomized and noisy) shared sample is exactly the local estimate of the histogram.  However, to ensure convergence in an almost sure sense we use techniques from stochastic approximation.  We identify three interesting regimes of behavior.  In the first, studied by Narayanan and Niyogi \cite{Narayanan11:lang}, agents converge to atomic distributions on a common single opinion.  In the second, agents converge to a common consensus estimate which in general is not equal to the true histogram.  Finally, we demonstrate a randomized protocol which, under mild technical assumptions, ensures convergence of agents' local estimates to the global histogram almost surely.  The stochastic approximation point of view suggests that a set of decaying weights can control the accumulation of noise along time and still compute the average histogram.

\subsection*{Related work}

In addition to the work in mathematical modeling of opinion formation~\cite{French:56power,Harary:59criterion,DeGroot:74consensus}, there has been a large body of work on consensus in terms of decision making initiated by Aumann~\cite{Aumann:76disagree}.  Borkar and Varaiya~\cite{BorkarV:82} studied distributed agreement protocols in which agents are trying to estimate a common parameter.  The agents randomly broadcast conditional expectations based on all of the information they have seen so far, and they  find general conditions under which the agents would reach an asymptotic agreement.  If the network is sufficiently connected (in a certain sense), the estimates converge to the centralized estimate of the parameter, even when the agents' memory is limited~\cite{TsitsiklisA:84}.  In these works the questions are more about whether agreement is possible at all, given the probability structure of the observation and communication.

There is a significant body of work on consensus and information aggregation in sensor networks~\cite{Fax:01thesis,FaxM:04flow,OlfatiSaberM:04consensus,AgaevC:00digraph,AgaevC:01spanning,ChebotarevA:02forest,5485031,5485032, BoydIT}.  From the protocol perspective, many authors have studied the effect of network topology on the rate of convergence of consensus protocols~\resp{\cite{MurraySurvey, Fax:01thesis,FaxM:04flow,OlfatiSaberM:04consensus,5485031,5485032,Fagn08,journals/siamco/OlshevskyT09,BoydIT}.}  For communication networks the speed can be accelerated by exploiting network properties \cite{DimakisSW:08gossip,BDTVPath,Tunc09,SarwateD:12mobility} (see surveys in \cite{MurraySurvey,Dimakis10gossipsurvey} for more references).  Others have studied how quantization constraints impact convergence \cite{Aysal07,NedicOOT:09dist,CarliBZ:10dynamic,KashyapBS:07quant,CarliFFZ:10quantgossip,ZhuM:11time,LavaeiM:10gossip, 5719290}.  However, in all of these works the agents are assumed to be some sort of computational devices like robotic networks or sensor networks.  A comprehensive view of this topic is beyond the scope of this paper. Instead, we focus on a few papers most relevant to our model and study: consensus with quantized messages and consensus via stochastic approximation. However, it is important to note that in contrast to all the studies discussed below,  our work primarily deals with an extension of the classic consensus (linear combination of private values) in that we are interested in  ensuring agreement over the space of distributions (histograms). 

Our goal in this paper is to ensure the convergence of each agent's local estimate to  a true and global discrete distribution via a low-overhead algorithm in which messages are chosen in a discrete set.  Our work is therefore related to the extensive recent literature on quantized consensus\cite{KashyapBS:07quant,CarliFFZ:10quantgossip,NedicOOT:09dist,LavaeiM:10gossip}.   In these works, as in ours, the communication between nodes is discretized (and in some cases the storage/computation at nodes as well~\cite{KashyapBS:07quant}) and the question is how to ensure consensus (within a bin) to the average.  This is in sharp contrast to our model which uses discrete messages to convey and ensure consensus on the network-wide histogram of discrete values.   As a result, in contrast to the prior work on quantization noise \cite{YildizS:08coding,Aysal07,CarliFFZ:10quantgossip}, the ``noise" is manufactured by our randomized sample selection scheme and hence plays a significantly different role.  

Our analysis uses similar tools from stochastic approximation as recent studies of consensus protocols~\cite{KarM:10quant,5719290,5575452}.  However, these works use stochastic approximation to address the effect of random noise in network topology, message transmission, and computation for a scalar consensus problem, while our use of standard theorems in stochastic approximation is to handle the impact of the noise that comes from the sampling scheme that generates our random messages. In other words, our noise is introduced by design even though our technique to control its cumulative effect is similar. 

\section{Model and Algorithms}


Let $[n]$ denote the set $\{1, 2, \ldots, n\}$ and let $\mbf{e}_i \in \mathbb{R}^M$ denote the $i$-th elementary row vector in which the $i$-th coordinate is $1$ and all other coordinates are $0$.  Let $\mbf{1}(\cdot)$ denote the indicator function and $\mathbf{1}$ the column vector whose elements are all equal to 1.  Let $\norm{\cdot}$ denote the Euclidean norm for vectors and the Frobenius norm for matrices.  We will represent probability distributions on finite sets as row vectors, \resp{ and denote the set of probability 
distributions on a countable set $A$ by $\mathbb{P}(A)$. } 

\subsection{Problem setup}

Time is discrete and indexed by $t \in \{0,1,2,\ldots\}$.  The system contains $n$ agents or ``nodes.'' At time $t$ the agents can communicate with each other according to an undirected graph $\mc{G}(t)$ with vertex set $[n]$ and edge set $\mc{E}(t)$.  Let $\mc{N}_i(t) = \{ j : (i,j) \in \mc{E}(t) \}$ be the set of neighbors of node $i$.  If $(i,j) \in \mc{E}(t)$ then nodes $i$ and $j$ can communicate at time $t$.  Let $G(t)$ denote the adjacency matrix of $\mc{G}(t)$ and let $D(t)$ be the diagonal matrix of node degrees.  The Laplacian of $\mc{G}(t)$ is $L(t) = D(t) - G(t)$.

At time $0$ every node starts with a single discrete \textit{sample} $X_i \in \mc{X} = [M]$.  The goal of the network is for each node to estimate the empirical distribution, or normalized histogram, of the observations $\{X_i : i \in [n] \}$:
	\begin{align*}
	\Pi(x) = \frac{1}{n} \sum_{i=1}^{n} \mbf{1}(X_i = x) \cdot \mbf{e}_x \qquad \forall x \in [M].
	\end{align*}

To make it simpler to characterize the overall communication overhead, we assume that 
	\begin{itemize}
	\item Agents can exchange messages $Y_i(t)$ lying in a finite set $\mc{Y}$.
	\item At each time $t = 1, 2, \ldots$ agents can transmit a single message to all of their neighbors.
	\end{itemize}
\resp{At each time $t = 0, 1, 2, \ldots$, each node $i$ maintains an internal estimate $Q_i(t)$ of the distribution $\Pi$, and we take $Q_i(0) = \mbf{e}_{X_i}$.   Every node $i$ generates its message $Y_i(t) \in \mc{Y}$ as a function of  this internal estimate $Q_i(t)$.  Furthermore, each node $i$ receives the messages $\{ Y_j(t) : j \in \mc{N}_i \}$ from its neighbors and use these messages to perform an update of its estimate $Q_i(t+1)$ using  $Q_i(t)$, $\{Y_j(t) : j \in \mc{N}_i \}$.} 

\resp{	
Since in a single time $t$ there are potentially $2 |\mc{E}(t)|$ messages transmitted in the network, a first approximation for the communication overhead of this class of  schemes is simply proportional to the number of edges in the network multiplied by the logarithm of the cardinality of set $\mc{Y}$.}

\resp{
We are interested in the properties of the estimates $Q_i(t)$ as $t \to \infty$.  In particular, we are interested in the case where every element in the set $\{Q_i(t) : i \in n\}$ converges almost surely to a common random variable $\mbf{q}^{*}$.  In this case, we call the random vector $\mbf{q}^{*}$ the consensus variable.  Different algorithms that we consider will result in different properties of this consensus variable.  For example, the we will consider the support of the distribution of $\mbf{q}^{*}$ as well as its expectation.  }

\subsection{Social sampling and linear updates}

\resp{
In this paper we assume that $\mc{Y} = \{\mbf{0}, \mbf{e}_1, \mbf{e}_1, \ldots \mbf{e}_M\}$, with 
the convention that node $i$ transmits nothing (or remains silent) when $Y_i(t) = \mbf{0}$. 
Furthermore, we
consider the class of schemes where the random message $Y_i(t) \in \mc{Y}$ of node $i$ at time $t$ is generated 
according to a distribution $P_i(t) \in \mathbb{P}(\mc{Y})$ which itself is a function of the estimate $Q_i(t)$.  
In other words, $P_i(t)$ is a row vector of length $M$ where $\prob(Y_i(t) = \mbf{e}_m) = P_{i,m}(t)$.}
We frequently refer to the random messages $Y_i(t) \in \mc{Y}, i\in[n], t=0,1, \ldots $ as \textit{social samples} because they correspond to nodes obtaining random samples of their neighbor's opinions.  \resp{ 
Note that although the random variable $Y_i(t)$ takes values in $\mathbb{R}^{M}$, it is supported only on the finite set $\mc{Y}$ and 
hence requires communicating $\log |\mc{Y}|$ information bits.} 

\resp{ For simplicity in this paper, we often rely on matrix representation across the network.  }
Accordingly, let $\mbf{Y}(t)$ be the $n \times M$ matrix whose $i$-th row is $Y_i(t)$.  Then we have $\expe[\mbf{Y}(t)] = \mbf{P}(t)$.  

\resp{ 
Let $\{W(t) : t = 0, 1, 2, \ldots\}$ be a sequence of $n \times n$ matrices with nonnegative entries, such that $W_{ij}(t) = 0$ for all $(i,j) \ne \mc{E}(t)$. }
 We study linear updates of the form
	\begin{align}
	Q_i(t+1) &= (1 - \delta(t) A_{ii}(t) ) Q_i(t) - \delta(t) B_{ii}(t) Y_i(t) \nonumber \\
		&\hspace{1in} + \sum_{j \in \mc{N}_i(t)} \delta(t) W_{ij}(t) Y_j(t).
		\label{eq:linearupdate}
	\end{align}
Here the parameter $\delta(t)$ is a step size for the algorithm.  Let $\mbf{Q}(t)$ be the $n \times M$ matrix whose $i$-th row is $Q_i(t)$.  We can write the iterates more compactly as
	\begin{align*}
	\mbf{Q}(t+1) &= (I - \delta(t) A(t)) \mbf{Q}(t) - \delta(t) B(t) \mbf{Y}(t) \nonumber \\
	&\hspace{1in} + \delta(t) W(t) \mbf{Y}(t),
	\end{align*}
where $A(t)$ and $B(t)$ are diagonal matrices.

In the next section we will analyze this update and identify conditions under which the estimates $Q_i(t)$ converge to a common 
$q^{\ast} \resp{ \in \mathbb{P}(\mc{Y})}$ and additional conditions under which $q^{\ast} = \Pi$.  To provide a unified analysis of these different algorithms, we transform the update equation into a stochastic iteration
	\begin{align}
	\mbf{Q}(t+1) = \mbf{Q}(t) + \delta(t) \left[ \bar{H}(t) \mbf{Q}(t) + \mbf{C}(t) + \mbf{M}(t) \right].
	\label{eq:compactupdate}
	\end{align}
In this form of the update, the matrix $\bar{H}(t)$ represents the mean effect of the network connectivity, $\mbf{M}(t)$ is a martingale difference term related to the randomness in the network topology and social sampling, and $\mbf{C}(t)$ is a correction term associated with the difference between the estimate $\mbf{Q}(t)$ and the sampling distribution $\mbf{P}(t)$.  

\begin{lemma}
The iteration in \eqref{eq:linearupdate} can be rewritten as \eqref{eq:compactupdate},
where 
\begin{align}
	\bar{H}(t) & \defeq  \bar{W}(t) - \bar{B}(t) - \bar{A}(t) \label{eq:hdef} \\
	\mbf{C}(t) & \defeq  \left( W(t) - B(t) \right) \left( \mbf{P}(t) - \mbf{Q}(t) \right) 
		\nonumber \\ & \hspace{0.12in}
		+ \left( W(t) - B(t) - \bar{W}(t) +\bar{B}(t) \right) \mbf{Y}(t) 
		\label{eq:perturbdef} \\
 \mbf{M}(t) &  \defeq \big( W(t) - B(t) - A(t)  
			- \bar{W}(t) + \bar{B}(t) + \bar{A}(t) \big) \mbf{Q}(t)
		\nonumber \\ 
		&\hspace{0.3in}
		+ \left( \bar{W}(t) - \bar{B}(t) \right) \left( \mbf{Y}(t) - \mbf{P}(t) \right)
		\nonumber \\ 
		& \hspace{0.3in}
		+ \left( \bar{W}(t) - \bar{B}(t) - W(t) + B(t) \right) \mbf{P}(t). \label{eq:MartingaleDiffdef} 
\end{align}
and the term $\mbf{M}(t)$ is a martingale difference sequence:
	\begin{align*}
	\expe[ \mbf{M}(t) | \mc{F}_t ] = \mbf{0}.
	\end{align*}
\end{lemma}

\begin{proof}
Rewriting the iterates, we see that
	\begin{align}
	\mbf{Q}(t+1) &= \mbf{Q}(t) + \delta(t) \big[ 
		- A(t) \mbf{Q}(t) 
		\notag \\
		&\hspace{1.2in}
		+ (W(t) - B(t)) \mbf{Y}(t)
		\big],
	\label{eq:iter_expand}
	\end{align}
and the term multiplied by $\delta(t)$ can be expanded:
	\begin{align}
	- A(t) \mbf{Q}(t) + & (W(t) - B(t)) \mbf{Y}(t) & \nonumber \\
		&\hspace{-0.5in}
	= \left( \bar{W}(t) - \bar{B}(t) - \bar{A}(t) \right) \mbf{Q}(t)  
		\nonumber \\
		&\hspace{-0.3in} 
		+ \big( W(t) - B(t) - A(t)  \notag \\
		&\hspace{0.8in} - \bar{W}(t) + \bar{B}(t) + \bar{A}(t) \big) \mbf{Q}(t)
		\nonumber \\ 
		& \hspace{-0.3in}
		+ \left( W(t) - B(t)  \right) \left( \mbf{Y}(t) - \mbf{Q}(t) \right) \nonumber \\
		&\hspace{-0.5in}
	= \left( \bar{W}(t) - \bar{B}(t) - \bar{A}(t) \right) \mbf{Q}(t) \notag \\
		&\hspace{-0.3in} 
		+ \big( W(t) - B(t) - A(t) 
		\notag \\
		&\hspace{0.8in}
		- \bar{W}(t) + \bar{B}(t) + \bar{A}(t) \big) \mbf{Q}(t)
		\nonumber \\ & \hspace{-0.3in}
		+ \left( W(t) - B(t) \right) \left( \mbf{P}(t) - \mbf{Q}(t) \right)
		\nonumber \\ & \hspace{-0.3in}
		+ \left( W(t) - B(t) \right) \left( \mbf{Y}(t) - \mbf{P}(t) \right)
	\nonumber \\
	&\hspace{-0.5in}
	= \left( \bar{W}(t) - \bar{B}(t) - \bar{A}(t) \right) \mbf{Q}(t)
		\nonumber \\ 
		& \hspace{-0.3in}
		+ \big( W(t) - B(t) - A(t) 
		\notag \\
		&\hspace{0.8in}- \bar{W}(t) + \bar{B}(t) + \bar{A}(t) \big) \mbf{Q}(t)
		\nonumber \\ & \hspace{-0.3in}
		+ \left( W(t) - B(t) \right) \left( \mbf{P}(t) - \mbf{Q}(t) \right)
		\nonumber \\ & \hspace{-0.3in}
		+ \left( \bar{W}(t) - \bar{B}(t) \right) \left( \mbf{Y}(t) - \mbf{P}(t) \right)
		\nonumber \\ & \hspace{-0.3in}
		+ \left( \bar{W}(t) - \bar{B}(t) - W(t) + B(t) \right) \mbf{P}(t)
		\nonumber \\ & \hspace{-0.3in}
		+ \left( W(t) - B(t) - \bar{W}(t) +\bar{B}(t) \right) \mbf{Y}(t)
		\big].
	\label{eq:extended}
	\end{align}
Define $\bar{H}$, $\mbf{C}$, and $\mbf{M}(t)$ as in \eqref{eq:hdef}, \eqref{eq:perturbdef} and \eqref{eq:MartingaleDiffdef}.  Furthermore, define
	\begin{align*}
	\mbf{M}(t) &= \big( W(t) - B(t) - A(t)  
		\nonumber \\ & \hspace{0.5in}
			- \bar{W}(t) + \bar{B}(t) + \bar{A}(t) \big) \mbf{Q}(t)
		\nonumber \\ & \hspace{0.3in}
		+ \left( \bar{W}(t) - \bar{B}(t) \right) \left( \mbf{Y}(t) - \mbf{P}(t) \right)
		\nonumber \\ & \hspace{0.3in}
		+ \left( \bar{W}(t) - \bar{B}(t) - W(t) + B(t) \right) \mbf{P}(t).
	\end{align*}
\resp{Taking conditional expectation of both sides of \eqref{eq:iter_expand} and noting that $\expe[ \mbf{Y}(t) | \mc{F}_t ] = \mbf{P}(t)$, we have the result. }
\end{proof}

Loosely speaking, the term $\mbf{C}(t)$ will be asymptotically vanishing if $\mbf{P}(t) \to \mbf{Q}(t)$ and the matrices $W(t)$ and $B(t)$ are asymptotically independent of $\mbf{Y}(t)$.  In the next section, we show that this stochastic approximation scheme converges under certain conditions on update rule.

\subsection{Example Algorithms \label{sec:examples}}

There are many algorithms which have the general form of the update rule \eqref{eq:linearupdate}. 
Before we proceed with the analysis of  \eqref{eq:linearupdate}, however, we look at three examples in which we can see interesting regimes of consensus-like behavior stemming from the update rule. 
In these numerical examples, the graph $\mc{G}(t)$ is a $5 \time 5$ grid, so the maximum degree of any node is $\dmax = 4$.  The initial values in the grid were drawn i.i.d. from the distribution $(0.4,0.3,0.2,0.1)$ on $M = 4$ elements.

\subsubsection{Averaging with social samples}  Suppose $\mbf{P}(t) = \mbf{Q}(t)$ for all $t$ and consider the update
	\begin{align}
	Q_i(t+1) = \frac{\dmax + 1 - d_i}{\dmax + 1} Q_i(t) + \sum_{j \in \mc{N}_i(t)} \frac{1}{\dmax + 1} Y_j(t).
	\label{eq:examp:singleton}
	\end{align}
This corresponds to $\delta(t) = 1$, $A_{ii}(t) = \frac{d_i}{\dmax + 1}$, $B_{ii}(t) = 0$, and $W_{ij}(t) = \frac{1}{\dmax + 1}$.  A trace of a single node's estimates for $M = 4$ is shown in Figure \ref{fig:examp:singleton}. The four lines correspond to the 4 elements of $Q_{i}(t)$.  As shown in \cite{Narayanan11:lang}, this procedure results in all $Q_i(t)$ converging to a consensus value that is a random singleton $\mbf{q}^*$ in $\{\mbf{e}_1,\mbf{e}_2, \ldots, \mbf{e}_M\}$ such that $\expe[\mbf{q}^*] = \Pi$. 
	\begin{figure}
	\centering
	\includegraphics[width=3in]{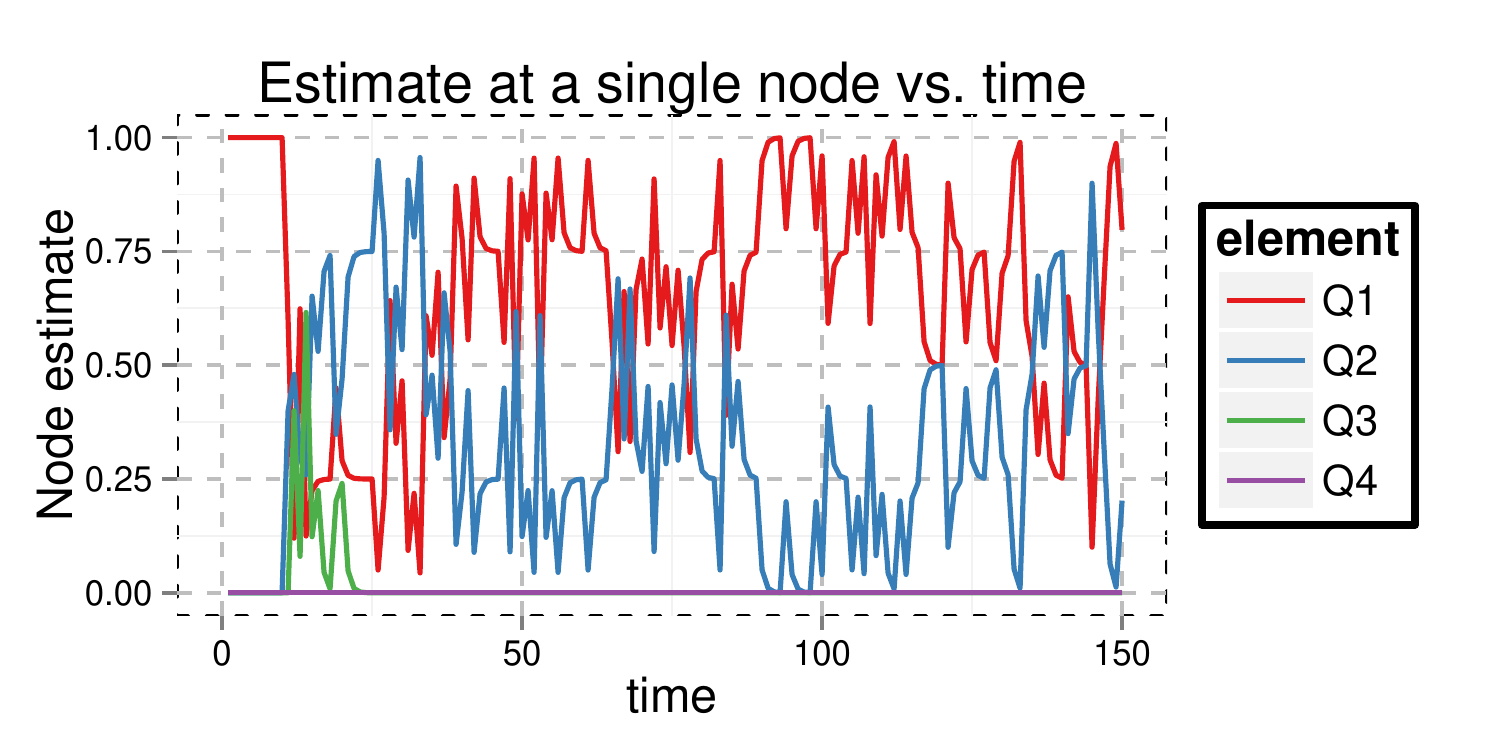}
	\caption{Trace of estimate of $Q_{i}(t)$ for a single node $i$ and $M = 4$, using the algorithm in \eqref{eq:examp:singleton} on a $5 \times 5$ grid graph.  The four lines correspond to the 4 entries of the vector $Q_i(t)$.  The estimates of all nodes converge to a random elementary vector $\mbf{q}^*\in \mc{Y}$; furthermore, $\mathbb{E}[\mbf{q}^*]=\Pi$. \label{fig:examp:singleton}}
	\end{figure}

\subsubsection{Averaging with social samples and decaying step size}  Suppose $\mbf{P}(t) = \mbf{Q}(t)$ for all $t$ and consider the update
	\begin{align}
	Q_i(t+1) &=  Q_i(t) - \frac{d_i}{\dmax + 1} \delta(t) Q_i(t)  \nonumber \\
	&\hspace{0.5in} + \delta(t) \sum_{j \in \mc{N}_i(t)} \frac{1}{\dmax+1} Y_j(t),
	\label{eq:examp:discount}
	\end{align}
with $\delta(t) = 10/(t+1)$.  This corresponds to $A_{ii}(t) = \frac{d_i}{\dmax + 1}$, $B_{ii}(t) = 0$, and $W_{ij}(t) = \frac{1}{\dmax + 1}$.  Figure \ref{fig:examp:discount} shows the estimates of all agents in a $5 \times 5$ grid tracking the estimate of $\Pi(m)$ for a single $m$.  The estimates of the agents converge to a consensus $\mbf{q}^*$ which is not equal to the true value $\Pi$ but $\expe[\mbf{q}^*] = \Pi$.  More generally, we will show that under standard assumptions on the step size and weights, if $\mbf{P}(t) = \mbf{Q}(t)$ then the iteration \eqref{eq:linearupdate} converges to a consensus state whose expectation is $\Pi$.
	\begin{figure}
	\centering
	\includegraphics[width=3in]{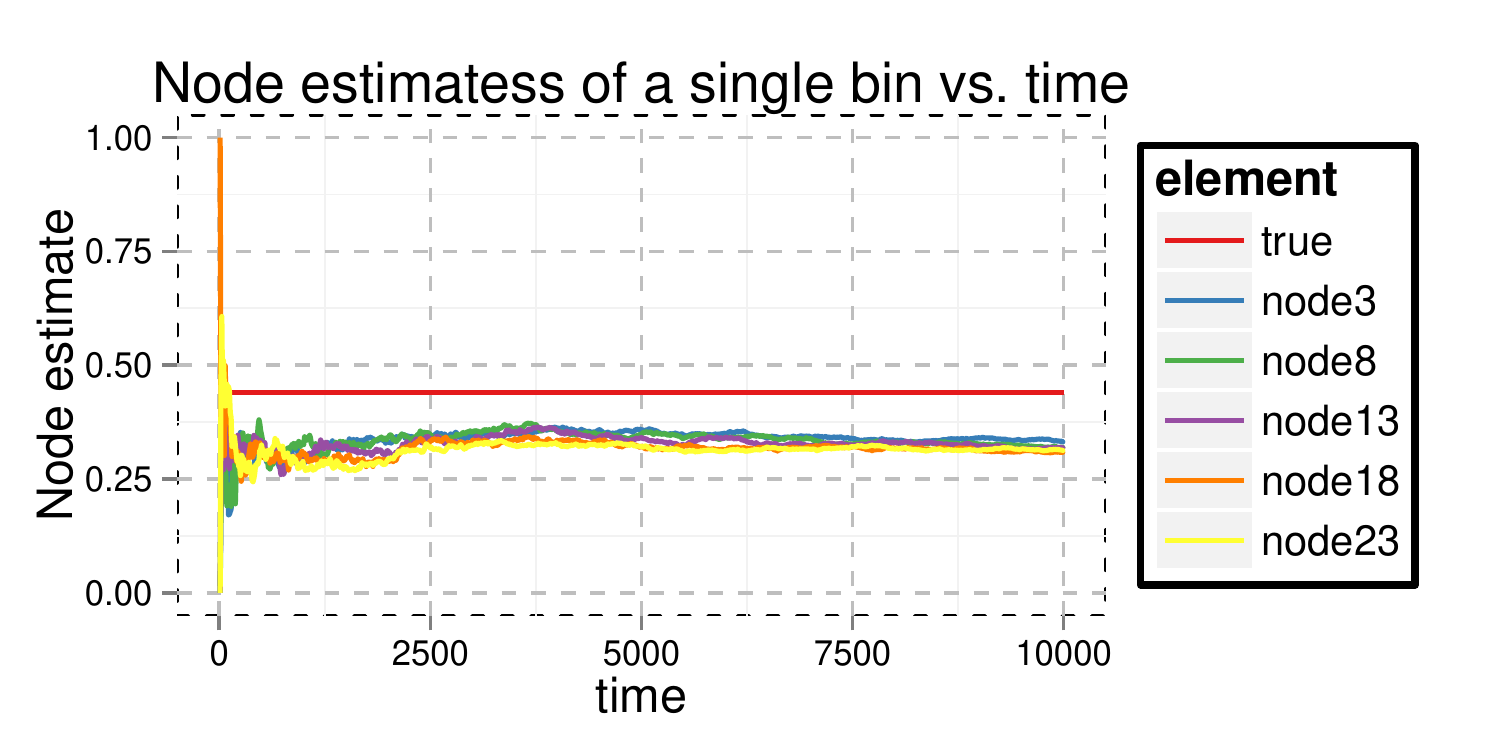}
	\caption{Trace of estimate of $Q_{i,m}(t)$ over all $i\in [n]$ for a single $m \in [M]$ with $M = 4$, using the algorithm in \eqref{eq:examp:discount} on a $5 \times 5$ grid graph.  The estimates of all nodes converge to a common random distribution 
	$\mbf{q}^* \in \mathbb{P}(\mc{Y})$ whose expectation is equal to the true distribution $\Pi$. \label{fig:examp:discount}}
	\end{figure}

\subsubsection{Exchange with social samples and censoring} Let $\delta(t) = 1/t$ and suppose that for each $(i,m)$, $m\neq \mbf{0}$, 
we set 
	\begin{align*}
	P_{i,m}(t) = \left\{ 
		\begin{array}{ll}
		0  & Q_{i,m}(t) < \dmax \delta(t) \\
		Q_{i,m}(t) & \mathrm{otherwise}
		\end{array}
		\right.
	\end{align*}
\resp{That is, under $\mbf{P}(t)$ node $i$ sends random messages corresponding to those elements of $\mbf{Q}(t)$ that are larger than $\dmax \delta(t)$ while it remains silent with probability $P_{i,0}(t) = 1 - \sum_{m \neq \mbf{0}} P_{i,m}(t)$.  Let $N_i(t) = \sum_{j \in \mc{N}_i} \mbf{1}( Y_j(t) \ne \mbf{0} )$ denote the total number of neighbors of node $i$ which did not remain silent at time $t$.  
Node $i$ updates its local estimate of the global histogram according to the following rule: }
	\begin{align}
	Q_i(t+1) = \left\{ 
		\begin{array}{ll}
		Q_i(t) & Y_i(t) = \mbf{0} \\
		Q_i(t) - \delta(t) N_i(t) Y_i(t) \\
		\hspace{0.3in} + \delta(t) \sum_{j \in \mc{N}_i} Y_j(t) & Y_i(t) \ne \mbf{0}
		\end{array}
		\right. .
	\label{eq:examp:censor}
	\end{align}
The behavior of \eqref{eq:examp:censor}, when $\delta(t) = \frac{10}{t+1}$ is illustrated in Figure \ref{fig:examp:censor}.  
The estimates $Q_i(t)$ of all agents converge to $\Pi$ almost surely under this update rule.  More generally, we show that given certain technical conditions, the mean across the agents of the sample paths $\mbf{Q}(t)$ is always equal to $\Pi$ and the estimate of each agent converges almost surely to $\Pi$.   In this case we  show that the rate of convergence is on the order of $1/t$.
	\begin{figure}
	\centering
	\includegraphics[width=3in]{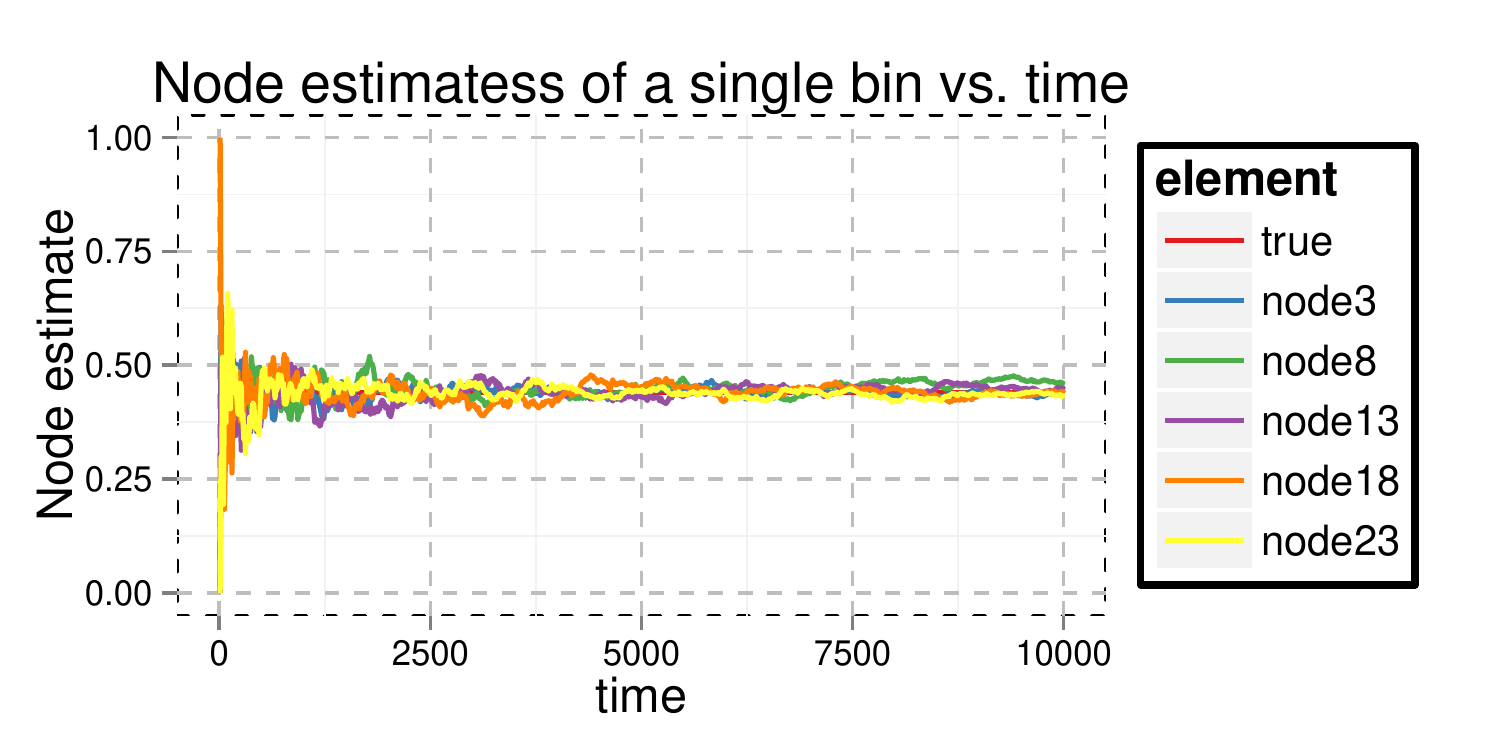}
	\caption{Trace of $Q_{i,m}(t)$ for a few different $i$ and a single $m \in M$ with $M = 4$, using the algorithm in \eqref{eq:examp:censor} on a $5 \times 5$ grid graph.  For this update rule, the estimates of all nodes converge almost surely to $\Pi$. \label{fig:examp:censor}}
	\end{figure}

These examples illustrate that the algorithms can display different qualitative behaviors depending on the choice of the step 
size $\delta(t)$ as well as the choice of $\mathbf{P}$.  In the first case, all agent estimates converge to a common random singleton, whereas under the second scenario they seem to converge to a common estimated histogram, even though this common estimated histogram might be far from the true histogram of the given initial values. 
  Finally, in the case where agents ``censor'' their low values and follow update rule (\ref{eq:examp:censor}), Figure 3 suggests an almost sure convergence to $\Pi$.
In the next section, we analytically confirm these empirical findings in with a unified analysis.

\section{Analysis}

We now turn to the analysis of the general protocol in \eqref{eq:linearupdate}.  We will need a number of additional conditions on the iterates in order to guarantee convergence.  Condition \ref{assume:coefficients} is that the agents compute convex combinations of their current estimates and the messages at time $t$.  This guarantees that the estimated distributions of the agents $Q_i(t)$ are proper probability distributions on $[M]$.

\subsection{Mean preservation}

Let $\mc{F}_t$ be the $\sigma$-algebra generated by 
	\[
	\{ \mbf{Q}_i(s) : s < t \} \cup \{ \mc{G}(s) : s < t \},
	\]
so $\mbf{Q}(t+1)$ is measurable with respect to $\mc{F}_t$.

\begin{assume}[Mixing coefficients]
\label{assume:coefficients}

For all $t > 0$ and all $i \in [n]$,
	\begin{align}
	\sum_{j \in \mc{N}_i} W_{ij}(t) - A_{ii}(t) - B_{ii}(t) = 0.
	\label{eq:assume:convex}
	\end{align}
\end{assume}

Let
	\begin{align*}
	\bar{A}(t) &= \expe[ A(t) | \mc{F}_t ] \\
	\bar{B}(t) &= \expe[ B(t) | \mc{F}_t ] \\
	\bar{W}(t) &= \expe[ W(t) | \mc{F}_t ].
	\end{align*}
Note that the coefficient $W(t)$ and $A(t)$ can, in general, depend on the messages $\mbf{Y}(t)$ as well as the graph $\mc{G}(t)$.

Our first result is a trivial consequence of the linearity of the update rule, and does not require any conditions beyond the fact that the estimate $\mbf{Q}$ is itself a distribution.

\begin{lemma}
\label{lem:meanpreserve}
Suppose Condition \ref{assume:coefficients} holds.  If $\mbf{P}(t) = \mbf{Q}(t)$ and $\mbf{W}(t)$ and $\mbf{B}(t)$ are independent of $\mbf{Y}(t)$, then for all $t$,
	\begin{align*}
	\expe[ \mbf{1}^{\trans} \mbf{Q}(t) ]  = \mbf{1}^{\trans} \mbf{Q}(0) = n \Pi.
	\end{align*} 
\end{lemma}

\begin{proof}
Given Condition \ref{assume:coefficients},
	\begin{align*}
	\expe[ \mbf{Q}(t+1) | \mc{F}_t] 
		&= (I - \delta(t) \bar{A}(t) ) \mbf{Q}(t) - \delta(t) \bar{B}(t) \mbf{Q}(t)  \nonumber \\
	&\hspace{1in} + \delta(t) \bar{W}(t) \mbf{Q}(t).
	\end{align*}
And therefore
	\begin{align} \label{martingale}
	\expe[\mbf{1}^{\trans} \mbf{Q}(t+1) | \mc{F}_t]  = \mbf{1}^{\trans} \mbf{Q}(t).
	\end{align}
On the other hand, since $\mbf{1}^{\trans} \mbf{Q}(0) = n \Pi$, the proof is complete. 
\end{proof}

This result is simple to see -- if the expected message $Y_i(t)$ is equal to $Q_i(t)$, then the mean of the dynamics are just those of average consensus.  However, it is not necessarily the case that the nodes converge to a consensus state, and if they do converge to a consensus state, that state may not be equal to $\Pi$ on almost every sample path.  The expected average of the node estimates will be equal to $\Pi$, and if they do reach a consensus the expected consensus value will be $\Pi$.  In this latter case it is sometimes possible to characterize the consensus value more explicitly.  For example, Narayanan and Niyogi \cite{Narayanan11:lang} show that in one version of 
this update rule, for all $i \in [n]$, $Q_i(t) \to \mbf{q}^* \in \mc{Y}-\{\mbf{0}\}$ and $\expe[ \mbf{q}^*] = \Pi$.  

\begin{lemma}[Singleton convergence \cite{Narayanan11:lang}]
\label{lem:singleton}
Suppose Condition \ref{assume:coefficients} holds.  If $\mbf{P}(t) = \mbf{Q}(t)$ and $W(t)=W$ and $B(t)=B$ are independent of time $t$ and (random) social samples $\mbf{Y}(t)$, then
	\begin{align*}
	\mathbb{P} \left\{\lim_{t\rightarrow \infty} \mbf{Q}(t)  = \mbf{1} \mbf{q}^*, \mbf{q}^* \in \mc{Y}-\{\mbf{0}\}  \right\}=1 .
	\end{align*} 
\end{lemma}

\subsection{Almost sure convergence}

The main result of this paper is obtaining sufficient conditions under which the update rule in \eqref{eq:linearupdate} converges to a state in which all agents have the same estimate of the histogram $\Pi$.  In general, the limiting state need not equal $\Pi$, but in some cases the process does converge almost surely to $\Pi$.  To show almost sure convergence we will need some additional conditions.  Condition \ref{assume:decr_stepsize} is a standard assumption from stochastic approximation on the step size used in the iteration.  A typical choice of step size is $\delta(t) = \Theta(1/t)$ which we used in the examples earlier.

\begin{assume}[Decreasing step size]
\label{assume:decr_stepsize}
The sequence $\delta(t) \to 0$ with $\delta(t) \ge 0$ satisfies
	\begin{align*}
	\sum_{t = 1}^{\infty} \delta(t) = \infty  \qquad \textrm{and} \qquad
	\sum_{t=1}^{\infty} \delta(t)^2 < \infty.
	\end{align*}
\end{assume}

\resp{ Condition \ref{assume:dynlimit} states that the expected weight matrices $\bar{H}(t)$ at each time are perturbations around a fixed time-invariant contraction matrix $\bar{H}$.  
This condition is satisfied in all three examples of interest above. Furthermore, it allows us to simplify the analysis. Note  that 
it seems to us that this condition is rather technical and can be relaxed at the cost of more cumbersome notation and complicated analysis. 
As a result, relaxing this assumption remains an area of future work. }

\begin{assume}[Limiting dynamics]
\label{assume:dynlimit}
There exists a symmetric matrix $\bar{H}$ such that
	\[
	\bar{H}(t) = \bar{H} + \tilde{H}(t),
	\]
where $\tilde{H}_{ij}(t) = O(\delta(t))$.  Furthermore, if $\lambda$ is an eigenvalue of $\bar{H}$ then we have $|\lambda| < 1$ and in particular $\bar{H} \mbf{1} = \mbf{0}$.  That is, $\bar{H}$ is a contraction.
\end{assume}

Condition \ref{assume:social_limit} implies that the perturbation term $\mbf{C}(t)$ in \eqref{eq:compactupdate} vanishes as the step size decreases.  This condition guarantees that the mean dynamics given by $\bar{H}$ govern the convergence to the final consensus state.

\begin{assume}[Bounded perturbation]
\label{assume:social_limit}
We have
	\[
	\norm{\expe[ \mbf{C}(t) | \mc{F}_t ] } = O(\delta(t)).
	\]
\end{assume}

Given these three conditions and the conditions on the coefficients we can show that the agent estimates in \eqref{eq:linearupdate} converge almost surely to a random common consensus state whose expected value, by Lemma \ref{lem:meanpreserve}, is equal to $\Pi$.  Thus for almost every sample path of the update rule, the estimates converge to a common value, but that value may differ across sample paths.  The expectation of the random consensus state is the true average.

\begin{theorem}
\label{thm:asConv}
Suppose Conditions \ref{assume:coefficients}, \ref{assume:decr_stepsize}, \ref{assume:dynlimit}, and \ref{assume:social_limit}  hold.  Then the estimate of each node $i$ governed by the update rule \eqref{eq:linearupdate} converges almost surely to a random variable 
$\mbf{q}^* \in \mathbb{P}(\mc{Y})$ which is a consensus state.  That is, $Q(t) \rightarrow \mbf{Q}^* = \mbf{1} \mbf{q}^*$ where
 $\expe[\mbf{Q}^*] = \mbf{1} \Pi$.
\end{theorem}

\begin{proof}
The result follows from a general convergence theorem for stochastic approximation algorithms~\cite[Theorem 5.2.1]{KushnerYin:2010}.  Define
	\[
	\mbf{V}(t) =  \bar{H}(t) \mbf{Q}(t) + \mbf{C}(t) + \mbf{M}(t).
	\]
Several additional conditions are needed to guarantee properties of $\mbf{V}(t)$ which ensure convergence of the update in \eqref{eq:compactupdate}.  Condition \ref{assume:decr_stepsize} guarantees that the step sizes decay slowly enough to take advantage of the almost sure convergence of stochastic approximation procedures.  The limit point is a fixed point of the matrix map $\bar{H}$ and Conditions \ref{assume:coefficients} and \ref{assume:dynlimit} show that this map is a contraction so the limit points are consensus states.  The final condition is that the noise in the updates can be ``averaged out.''  This follows in part because the process is bounded, and in part because Condition \ref{assume:social_limit} shows that the perturbation must be decaying sufficiently fast.

We must verify a number of conditions~\cite[p.126]{KushnerYin:2010} to use this theorem.
	\begin{enumerate}
	\item Condition \ref{assume:decr_stepsize} shows that the step sizes not summable but are square summable~\cite[(5.1.1) and (A2.4)]{KushnerYin:2010}.
	\item The iterates are bounded in the sense that $\sup_{t} \expe[ \norm{\mbf{V}(t)}^2 ] < \infty$. 
		 This follows because $Q_i(t)$ is a probability distribution for all $t$, so the updates 
		 must also be bounded~\cite[(A2.1)]{KushnerYin:2010}.
	\item If we take the expected update
	\begin{align*}
	\expe[ \mbf{V}(t) | \mc{F}_t ] 
		&= \bar{H}(t) \mbf{Q}(t) + \expe[ \mbf{C}(t) | \mc{F}_t ] \\
		&= \bar{H} \mbf{Q}(t) + \tilde{H}(t) \mbf{Q}(t) + \expe[ \mbf{C}(t) | \mc{F}_t ],
	\end{align*}
so we can write this conditional expectation as the sum of a measurable function $\bar{H} \mbf{Q}(t)$ and a random yet diminishing perturbation $\tilde{H}(t)  \mbf{Q}(t) + \expe[ \mbf{C}(t) | \mc{F}_t ]$.  Furthermore, from Condition \ref{assume:dynlimit}, the map $\bar{H}$ is continuous~\cite[(A2.2)-(A2.3)]{KushnerYin:2010}.
	\item The final thing to check~\cite[(A2.5)]{KushnerYin:2010} is that the random perturbation in the expected update decays sufficiently quickly:
	\begin{align*}
	\sum_{t=1}^{\infty} \delta(t) \norm{ \tilde{H}(t)  \mbf{Q}(t) + \expe[ \mbf{C}(t) | \mc{F}_t ] }
	&\\
	&\hspace{-1.9in} \le \sum_{t=1}^{\infty} \delta(t) \norm{ \tilde{H}(t) } \norm{ \mbf{Q}(t)} 
		 + \sum_{t=1}^{\infty} \delta(t) \norm{\expe[ \mbf{C}(t) | \mc{F}_t ] } \\
	&\hspace{-1.9in} < \infty.
	\end{align*}
The last step follows from Conditions \ref{assume:decr_stepsize} and \ref{assume:social_limit} \resp{as well as boundedness of 
$ \mbf{Q}(t)$}.	\end{enumerate}
Applying Theorem 5.2.1 of Kushner and Yin \cite{KushnerYin:2010} shows that the estimates converge to a limit set of the linear map $\bar{H}$.  Furthermore, from Condition \ref{assume:dynlimit} we know $\bar{H}$ is a contraction with a single eigenvector at $\mbf{1}$. In other words, the limit points are of the form $\mbf{Q}^{\ast} = \mbf{1} \mbf{q}^{\ast}$ where every row is identical.
\end{proof}

The preceding theorem shows that the updates converge almost surely to a limit when the step sizes are decreasing, even though as \resp{shown in \cite{Narayanan11:lang}}, we know that decreasing step size is not necessary for almost sure convergence.  So far the algorithm has no provable advantage to that of \cite{Narayanan11:lang}, in that each node's estimate converges to a 
consensus state  $\mbf{q}^{\ast}$, but $\mbf{q}^{\ast}$ need not equal $\Pi$.  However, by ensuring that the sample path of the algorithm is ``mean preserving''  (the sum of the $j$-th components of all $Q_i(t)$'s is equal to $\Pi_j$), this consensus limit becomes equal to $\Pi$.  

\begin{assume}[Mean preservation]
\label{assume:mean}
The average of the node estimates is $\Pi$
	\[
	\mbf{1}^{\trans} \mbf{Q}(t) = \Pi \qquad \forall t.
	\]
\end{assume}

\begin{corollary}
\label{cor:censor}
Suppose Conditions \ref{assume:coefficients}, \ref{assume:decr_stepsize}, \ref{assume:dynlimit}, \ref{assume:social_limit}, and \ref{assume:mean} hold. Then $ \mbf{Q}(t) \to \mbf{Q}^{\ast}$ almost surely, where $\mbf{Q}^{\ast} = \mbf{1} \Pi$ almost surely.
\end{corollary}

\subsection{Rate of convergence}

We now turn to bounds on the expected squared error of in the case where $\mbf{Q}(t) \to \mbf{Q}^{\ast}$ almost surely.  

\begin{theorem}[Rate of convergence]
\label{thm:rate}
Suppose that Conditions \ref{assume:coefficients},  \ref{assume:decr_stepsize}, \ref{assume:dynlimit}, and \ref{assume:social_limit} also hold.  Then there exists a constant $C$ such that
	\begin{align*}
	\expe\left[ \norm{ \mbf{Q}(t) - \mbf{Q}^{\ast} }^2 \right] \le C \delta(t).
	\end{align*}
\end{theorem}

\begin{proof}
First note that in the process \eqref{eq:linearupdate}, $Q_i(t)$ is a probability distribution, so the entire process lies in a bounded compact set, and under Conditions \ref{assume:dynlimit} and \ref{assume:social_limit} we can write the iteration as
	\begin{align*}
	\mbf{Q}(t+1) = \mbf{Q}(t) + \delta(t) \left[ \bar{H} \mbf{Q}(t) + \mbf{D}(t) + \mbf{M}(t) \right],
	\end{align*}
where the perturbation term $\norm{ \expe[\mbf{D}(t) | \mc{F}_t ] } = O(\delta(t))$.  We can now apply Theorem 24 of  Benveniste et al.~\cite[p. 246]{BMP:adaptive}, which requires checking similar conditions as the previous Theorem.
	\begin{enumerate}
	\item Condition \ref{assume:decr_stepsize} shows the step sizes are not summable~\cite[(A.1)]{BMP:adaptive}.
	\item Treat the tuple of random variables $(\mbf{A}(t), \mbf{B}(t), \mbf{W}(t), \mbf{Y}(t))$ as a state variable $\mbf{S}(t)$.  This state is measurable with respect to $\mc{F}_t$, and there exists a conditional probability distribution corresponding to the update~\cite[(A.2)]{BMP:adaptive}.
	\item Let
		\[
		\mbf{N}(t) =  \mbf{M}(t) + \resp{\mbf{C}(t) }- \expe[\mbf{D}(t) | \mc{F}_t ],
		\]
so $\mbf{N}(t)$ is still a martingale difference.  If we define $\mbf{J}(t) = \frac{1}{\delta(t)} \expe[\resp{\mbf{C}(t)} | \mc{F}_t ]$ we can rewrite the iterates as
		\begin{align*}
		\mbf{Q}(t+1) = \mbf{Q}(t) + \delta(t) \left[ \bar{H} \mbf{Q}(t) + \mbf{N}(t) \right] + \delta(t)^2 \mbf{J}(t).
		\end{align*}
	Again, the terms $\norm{ \bar{H} \mbf{Q}(t) + \mbf{N}(t) }$ and $\norm{ \mbf{J}(t) }$ are bounded by a constant~\cite[(A.3) and (A.5)]{BMP:adaptive}.
	\item Since $\bar{H}$ is a linear contraction, it is also Lipschitz, and the martingale difference $\mbf{N}(t)$ is bounded, which implies condition (A.4) of Benveniste et al.~\cite[p. 216]{BMP:adaptive}.
	\end{enumerate}
With the validity of the above conditions , the assertion of the theorem follows directly \cite[p. 216]{BMP:adaptive}:
	\begin{align*}
	\expe\left[ \norm{ \mbf{Q}(t) - \mbf{Q}^{\ast} }^2 \right] = O(\delta(t)). 
	\end{align*}
\end{proof}


\subsection{Example Algorithms Revisited}

We can now describe how the results apply to the algorithms described in section \ref{sec:examples}.

\subsubsection{Averaging with social samples}  Our first algorithm in \eqref{eq:examp:singleton} was one in which the nodes perform a weighted average of the distribution of the messages they receive with their current estimate.  The general form of the algorithm was
	\[
	\mbf{Q}(t+1) = \left(I - A \right) \mbf{Q}(t) + W \mbf{Y}(t),
	\]
which corresponds to choosing $\delta(t) = 1$ and $B(t) = 0$.  For the specific example in \eqref{eq:examp:singleton}, $A = \frac{1}{\dmax + 1} D$ and $W = \frac{1}{\dmax + 1} G$, where $G$ is the adjacency matrix of the graph and $D$ is the diagonal matrix of degrees.  Furthermore, $\mbf{P}(t) = \mbf{Q}(t)$ for all $t$.   
	\begin{align}
	\mbf{Q}(t+1) = \mbf{Q}(t) + (W - A) \mbf{Q}(t) + W ( \mbf{Y}(t) - \mbf{Q}(t) ).
	\label{eq:singletonLap}
	\end{align}
The term $W - A$ is the graph Laplacian of the graph with edge weights given by $W$.  The following is is a corollary of Lemma \ref{lem:meanpreserve} and Lemma \ref{lem:singleton}.

\begin{corollary}
For the update given in \eqref{eq:singletonLap}, the estimates $\mbf{Q} \to \mbf{Q}^*$ almost surely, where $\mbf{Q}^*$ is a random matrix taking values in the set $\{ \mbf{1} \mbf{q}^{*} : \mbf{q}^{*} \in \mc{Y} \}$ such that $\expe[\mbf{Q}^*] = \mbf{1} \Pi$.
\end{corollary}

Examining \eqref{eq:singletonLap}, we see that the the Laplacian term drives the iteration to a consensus state, but the only stable consensus states are those for which $\mbf{Y}(t) - \mbf{Q}(t) = \mbf{0}$, which means $\mbf{Y}(t)$ must be equal to $\mbf{Q}(t)$ almost surely.  This means each row of $\mbf{Q}(t)$ must correspond to a degenerate distribution of the form $\mbf{e}_m$.  

\subsubsection{Averaging with social samples and decaying step size}  The second class of algorithms, exemplified by \eqref{eq:examp:discount}, has the following general form:
	\[
	\mbf{Q}(t+1) = \left( 1 - \delta(t) A \right) \mbf{Q}(t) + \delta(t) W \mbf{Y}(t),
	\]
and again $\mbf{P}(t) = \mbf{Q}(t)$ for all $t$.  This is really the same as \eqref{eq:singletonLap} but with a decaying step size $\delta(t)$:
	\begin{align}
	\mbf{Q}(t+1) &= \mbf{Q}(t) + \delta(t) (W - A) \mbf{Q}(t) 
	\nonumber \\ & \hspace{1in}
		+ \delta(t) W ( \mbf{Y}(t) - \mbf{Q}(t) ).
	\label{eq:examp:discountLap}
	\end{align}
However, the existence of a decreasing step size means that the iterates under this update behave significantly differently than those governed by \eqref{eq:singletonLap}. The convergence of this algorithm is characterized by Theorems \ref{thm:asConv} and \ref{thm:rate}.
	
\begin{corollary}
For the update given in \eqref{eq:examp:discountLap} with $\delta(t) = 1/t$, the estimates $\mbf{Q}(t) \to \mbf{Q}^*$ almost surely, where $\mbf{Q}^*$ is a random matrix in the set $\{ \mbf{1} \mbf{q}^{\trans} : \norm{\mbf{q}}_1 = 1, q_m > 0 \}$ and $\expe[\mbf{Q}^*] = \mbf{1} \Pi^{\trans}$.  Furthermore, $\expe[ \norm{ \mbf{Q}(t) - \mbf{Q}^*}^2 ] = O(1/t)$.
\end{corollary}

\subsubsection{Exchange with social samples and censoring}  The last algorithm in \eqref{eq:examp:censor} has a more complex update rule, but it is a special case of the generic update
	\begin{align}
	\mbf{Q}(t+1) = \mbf{Q}(t) - \delta(t) B(t) \mbf{Y}(t) + \delta(t) W(t) \mbf{Y}(t).
	\label{eq:examp:censorGeneral}
	\end{align}
For a fixed weight matrix $W$, define the the thresholds $\Delta(t) = \delta(t) \sum_{j \in \mc{N}_i} W_{ij}$ and the sampling distribution $P_{i,m}(t) = \mbf{1}( Q_{i,m}(t) > \Delta(t) )$.  The social samples $\mbf{Y}(t)$ are sampled according to this distribution and the weight matrices are defined by
	\begin{align*}
	W_{ij}(t) &= \left\{
		\begin{array}{ll}
		0 & i \ne j, Y_{i} = \mbf{0} \ \mathrm{or}\  Y_{i} \ne \mbf{0} \\
		W_{ij} & i \ne j, Y_{i} = \mbf{0} \ \mathrm{and}\  Y_{i} \ne \mbf{0}
		\end{array}
		\right. \\
	B_{ii}(t) &= \sum_{j \in \mc{N}_i} W_{ij}(t)
	\end{align*}
In this algorithm the iterates keep $\sum_{i} Q_{i,m}(t)$ constant over time by changing the sampling distribution $\mbf{P}(t)$ over time and by using the weight matrix $B(t)$ to implement a ``mass exchange'' policy between nodes.  At each time, agent $i$ samples a opinion $Z_i(t)$.  If $Q_{i,Z_i(t)}(t)$ is large enough, the agent sends $Y_i(t) = Z_i(t)$, giving $\delta(t) W_{ij}$ mass to each neighbor $j$ and subtracting the corresponding mass from its own opinion.  If $Q_{i,\mbf{Z}_i(t)}(t)$ is not large enough it exchanges nothing with its neighbors.  The distribution $\mbf{P}(t)$ implements this ``censoring'' operation.  By keeping the total sum on each opinion fixed, Corollary \ref{cor:censor} shows that the estimates converge almost surely to $\Pi$.

\begin{corollary}
For the update given in \eqref{eq:examp:censorGeneral} with $\delta = 1/t$, the estimates $\mbf{Q}(t) \to \mbf{1} \Pi$ almost surely, and $\expe[ \norm{ \mbf{Q}(t) - \mbf{1} \Pi}^2 ] = O(1/t)$.

\end{corollary}

\section{Empirical results}

\begin{figure*}
\centering
\subfigure[Grid]{
\includegraphics[width=2.8in]{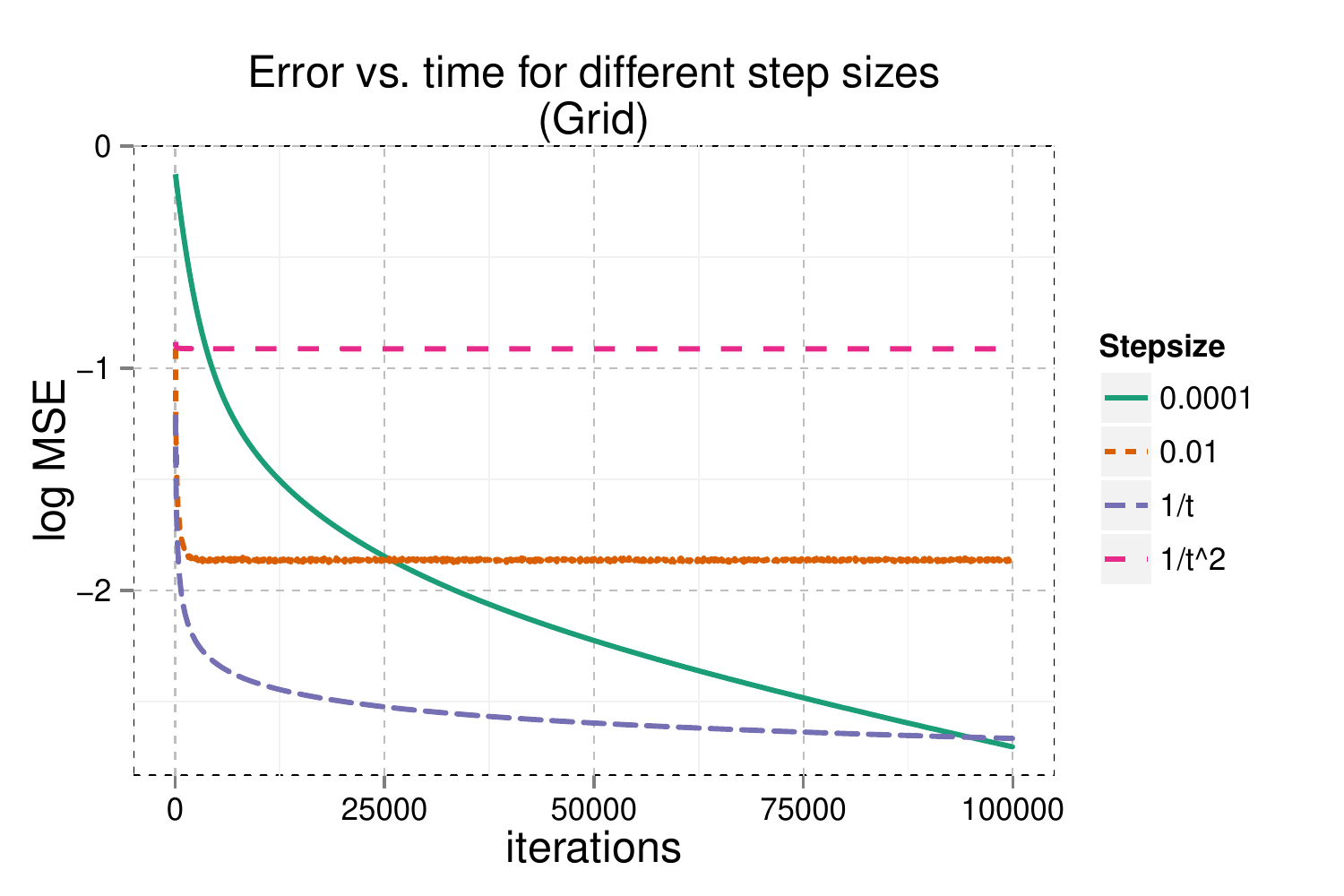}
\label{fig:evt1}
}
\subfigure[Preferential Attachment]{
\includegraphics[width=2.8in]{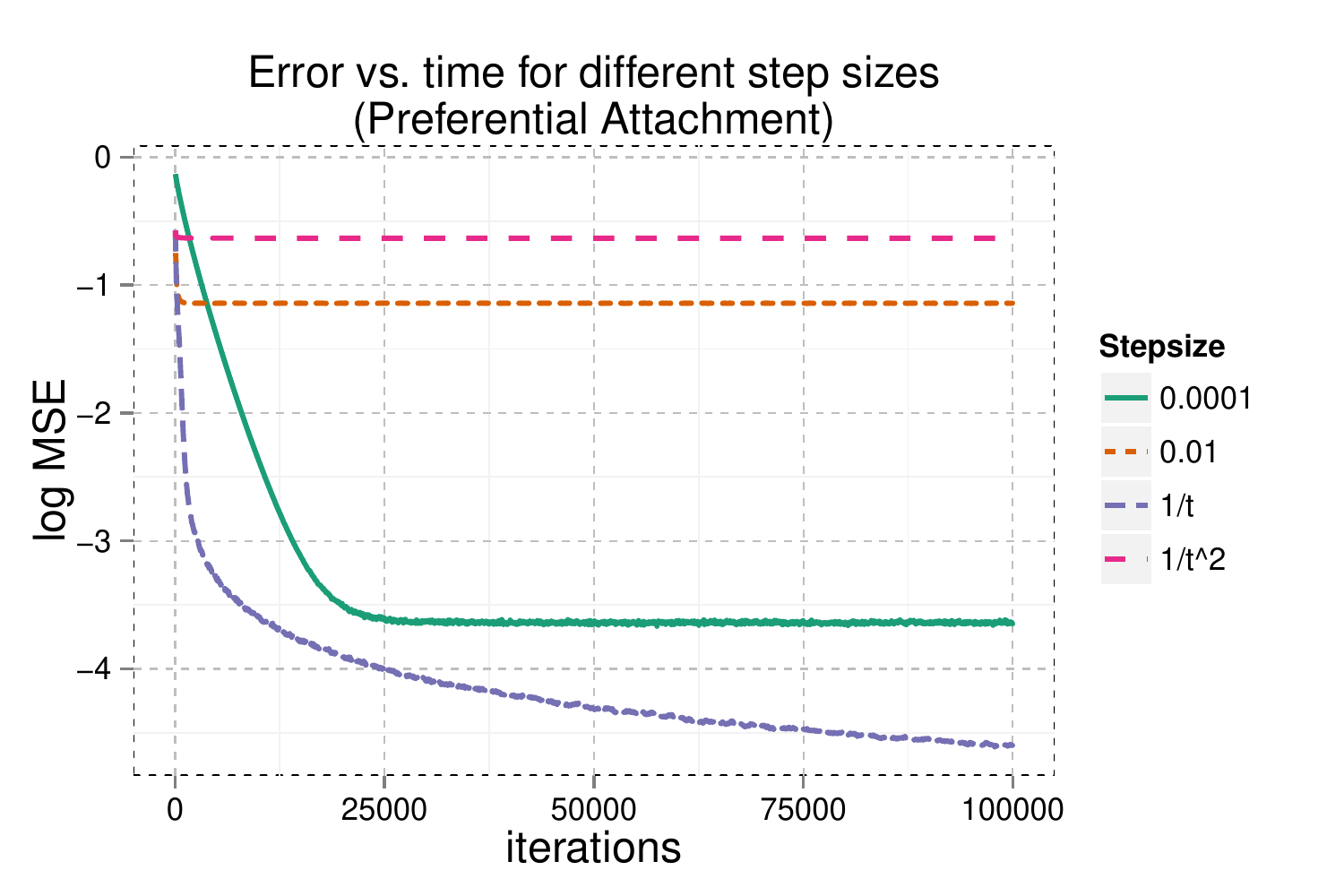}
 \label{fig:evt2}
} \\
\subfigure[Watts-Strogatz]{
\includegraphics[width=2.8in]{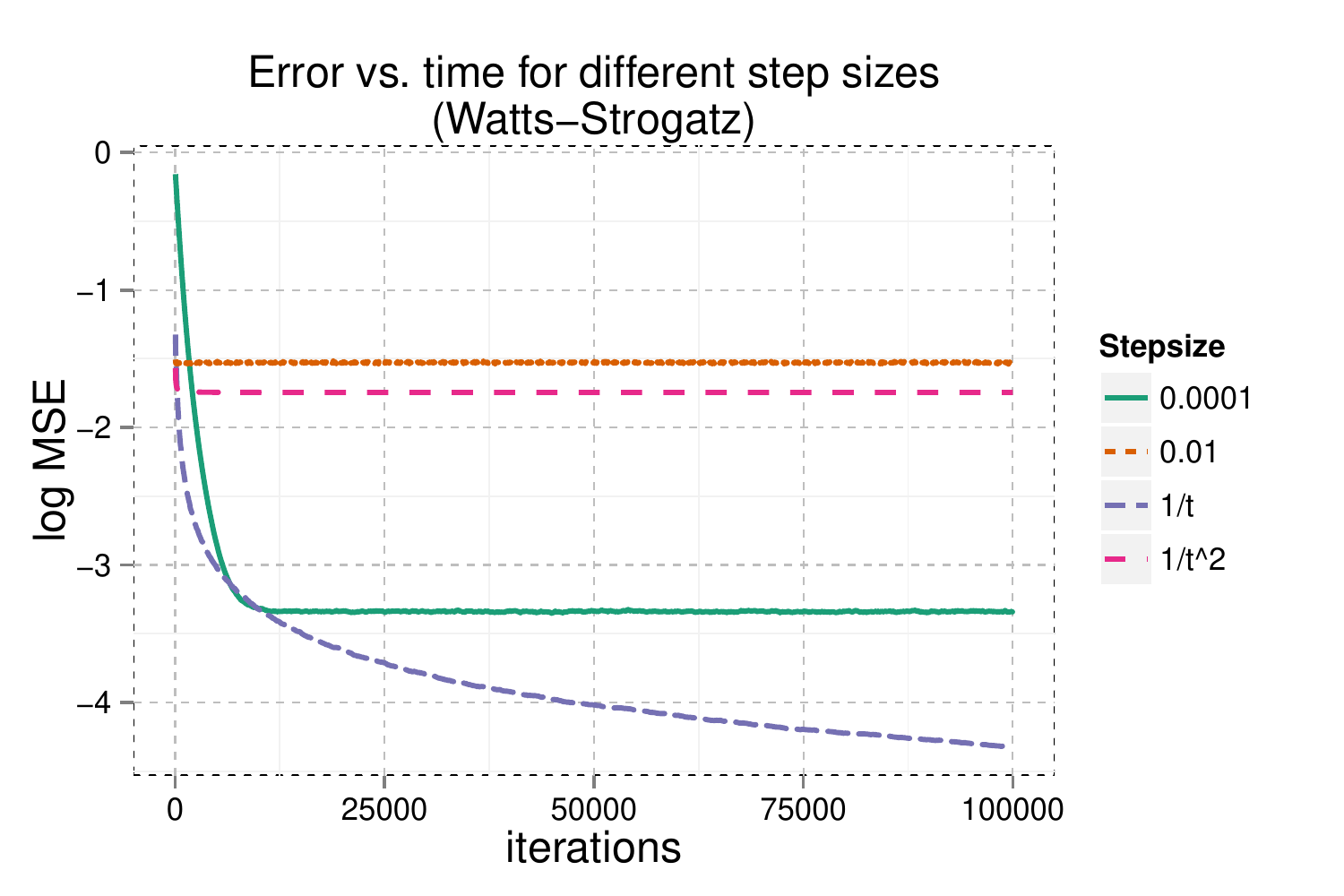}
\label{fig:evt3}
}
\subfigure[Star]{
\includegraphics[width=2.8in]{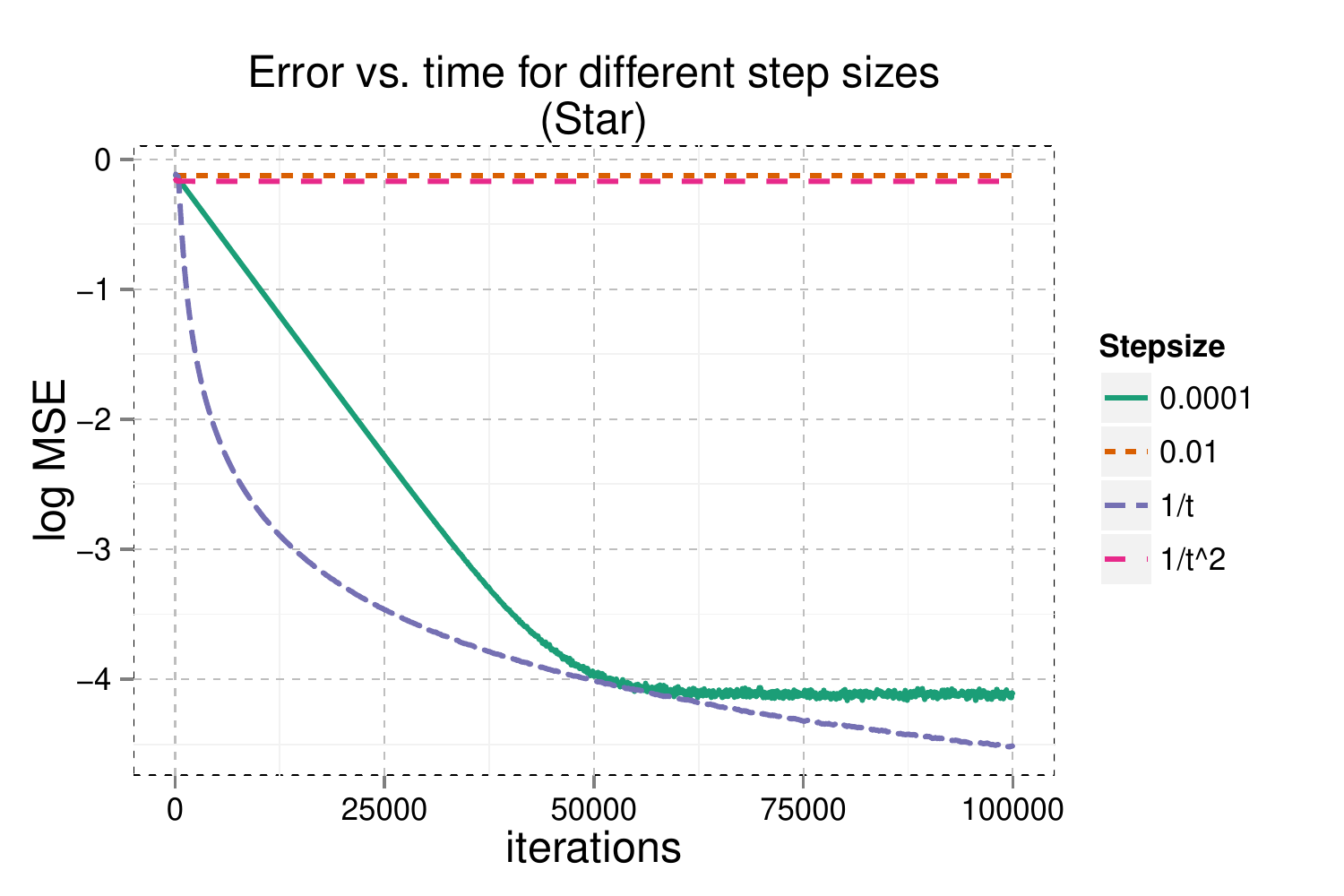}
\label{fig:evt4}
}
\caption{Average MSE between the agent estimates and the true histogram initial node values versus time for 4 different 100-node graphs : a $10 \times 10$ grid, preferential attachment graphs with three edges generated per new node, Watts-Strogatz graphs with rewiring probability $0.1$, and a star with one central node and 99 peripheral nodes. \label{fig:evt}}
\end{figure*}

The preceding analysis shows the almost-sure convergence of all node estimates for some social sampling strategies, and in some cases characterizes the rate of convergence.  However, the analysis does not capture the effect of problem parameters such as the initial distribution of node values and the network topology.  These factors are well known to affect the rate of convergence of many distributed estimation and consensus procedures -- in this section we provide some empirical results about these effects.

We considered a number of different topologies for our simulations:
	\begin{itemize}
	\item The $\sqrt{n} \times \sqrt{n}$ grid has vertex set $[\sqrt{n}] \times [\sqrt{n}]$ and edge exists between two nodes whose $L_1$ distance is 1.
	\item A star network has a single central vertex which is connected by a single edge to $n - 1$ other vertices.
	\item An Erd\H{o}s-R\'{e}nyi graph~\cite{Erdos:1960} on vertex set $[n]$ contains each edge $(i,j)$ independently with probability $p$.  We choose $p = 0.6$.
	\item A preferential attachment graph \cite{barabasi1999emergence,BollobasR04} is constructed by adding vertices one at a time.  A new vertex is connected to an existing vertex with a probability that is a function of the current degree of the vertices.  We allowed each new vertex to be connected to $3$ preceding vertices.
	\item A 2-dimensional Watts-Strogatz graph is a grid with randomly ``rewired'' edges~\cite{watts1998collective}.  We chose rewiring probability $0.1$.
	\end{itemize}
Details on the random graph models can be found in the \texttt{igraph} package for the \texttt{R} statistical programming language \cite{igraph}. 
In the simulations we calculated the average of $\frac{1}{n} \norm{ \mbf{Q}(t) - \mbf{1} \Pi}^2$ across runs of the simulation, which is the average mean-squared-error (MSE) per node of the estimates.

\subsection{Network size and topology}

We were interested in finding how the convergence time depends on the step size and network topology.  To investigate this we simulated the grid, preferential attachment, Watts-Strogatz, and star topologies described above on networks of $n = 100$ nodes with $M = 5$.  The initial node values were drawn i.i.d. from a distribution $(0.1, 0.25, 0.15, 0.3, 0.2)$ on 5 elements.
Simulations were averaged over 100 instances of the network and initial values.

Figure \ref{fig:examp:censor} shows that the estimates converge almost surely to the true histogram $\Pi$ when $\delta(t) = 1/t$.  While our theoretical analysis was for this case, in practice stochastic optimization is often used with a constant step size because the algorithm converges faster to a neighborhood of the optimal solution when $\delta(t)$ is appropriately small.  In order to assess if this is the case in our model, we simulated variants of the algorithm in \eqref{eq:examp:censor} with different settings for $\delta(t)$.

Figure \ref{fig:evt} shows the error between the local estimates and the true histogram of initial node values under four different topologies and four different choices for $\delta(t)$.  For $\delta(t) = 1/t$ the algorithm satisfies the conditions of the theorem and we can see the rather rapid convergence to the mean.  If $\delta(t)$ is constant, then the error does not converge to $0$ but can still be quite small if the step size is small.  This is similar to the fixed-weight algorithm with a weight matrix that has very small off-diagonal entries.  By contrast, the weight sequence $\delta(t) = 1/t^2$ decays too quickly and there is a large residual MSE.  

\resp{We see a greater effect on the convergence time by looking at different graph topologies for the same number of nodes.  For graphs with fast mixing time such as the preferential attachment and Watts-Strogatz model, the error decreases much more rapidly than for the grid or star.  This suggests that the mixing time of the random walk associated with the weight matrix of the algorithm should affect the rate of convergence, as is the case in other consensus algorithms. } The effect of choosing weight sequences that do not guarantee almost sure convergence also varies depending on the network topology.  For sparsely connected networks like the star, the performance is quite poor unless the weight sequence is chosen appropriately.  However, for denser networks like the Watts-Strogatz model, the difference may be more modest.

\subsection{The effect of the initial distribution}

\begin{figure}
\centering
\includegraphics[width=2.6in]{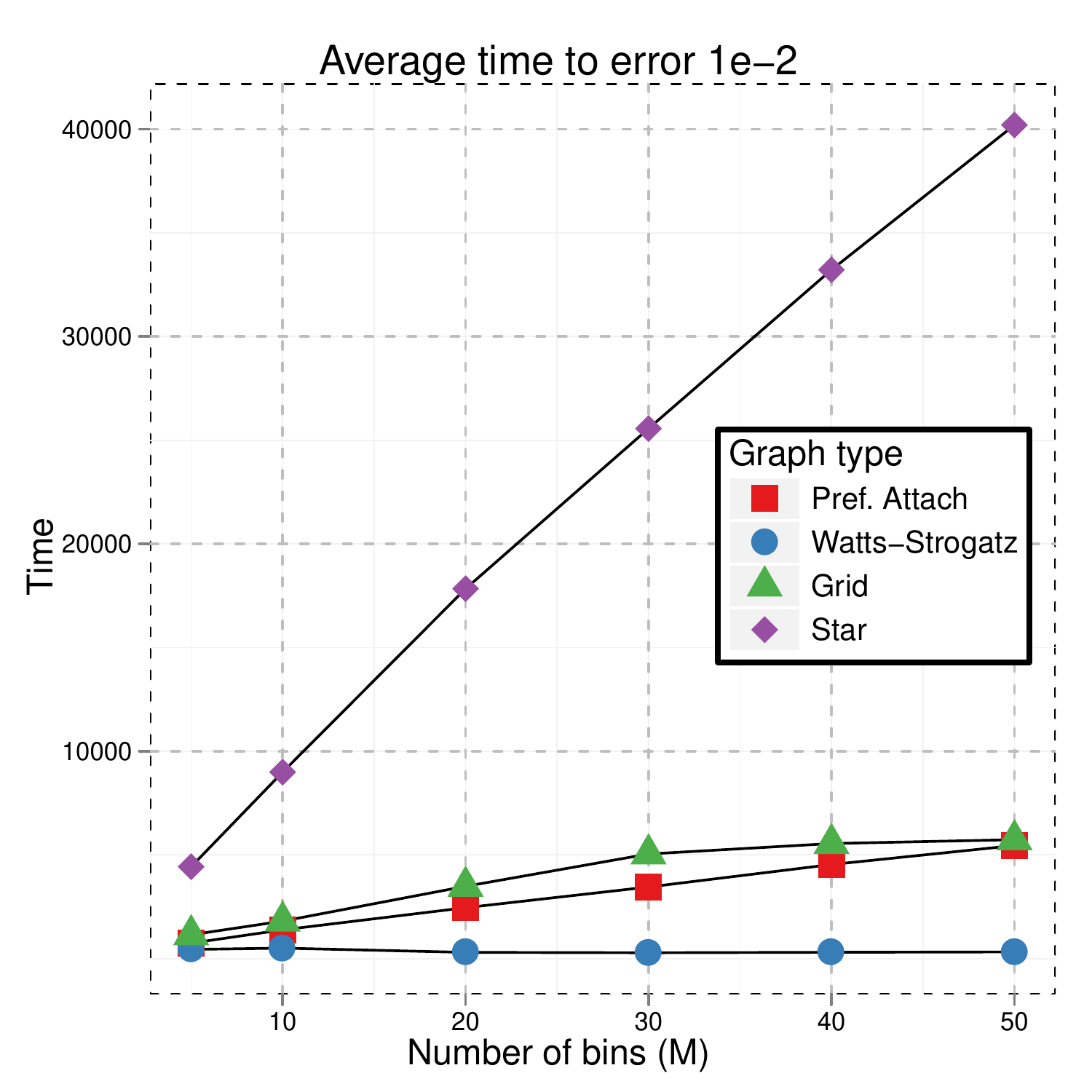}
\caption{Time to get to MSE of $10^{-2}$, averaged across nodes, versus $M$ for a uniform distribution on the four different network topologies with $n = 100$ nodes.  \label{fig:tvm2}}
\end{figure}

\begin{figure*}
\centering
\subfigure[Erd\H{o}s-R\'{e}nyi]{\includegraphics[width=2.6in]{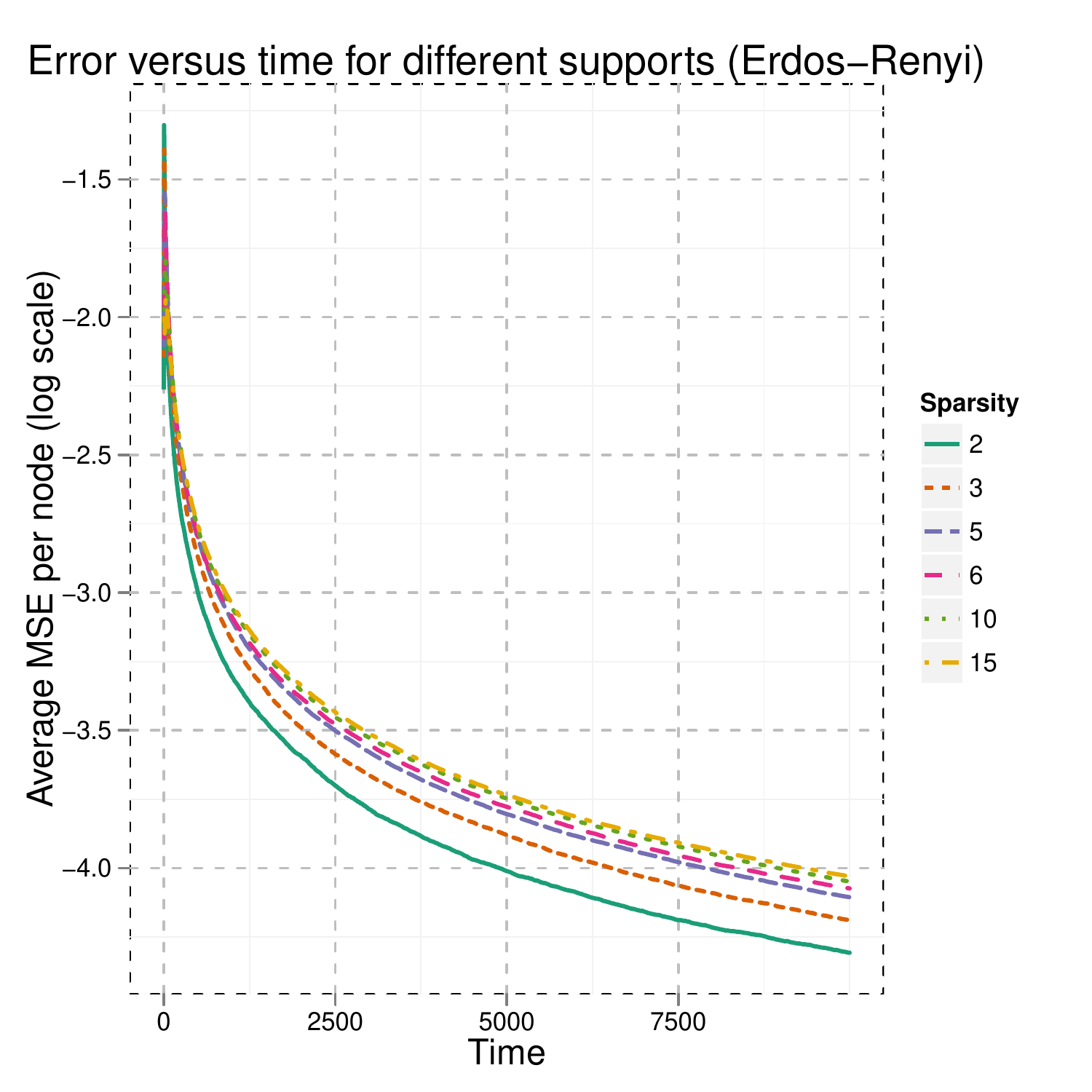}}%
\subfigure[Grid]{\includegraphics[width=2.6in]{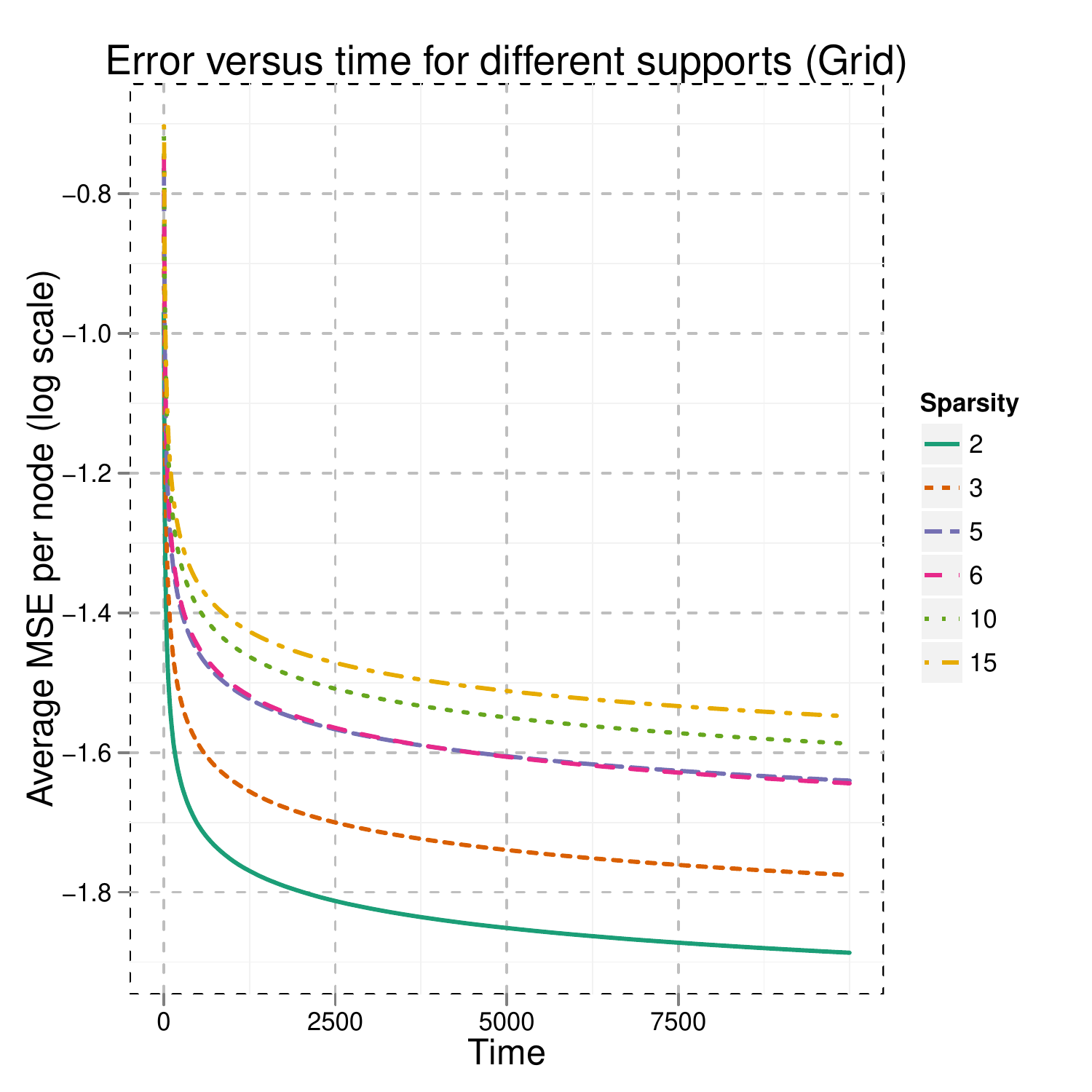}} \\
\subfigure[Preferential attachment]{\includegraphics[width=2.6in]{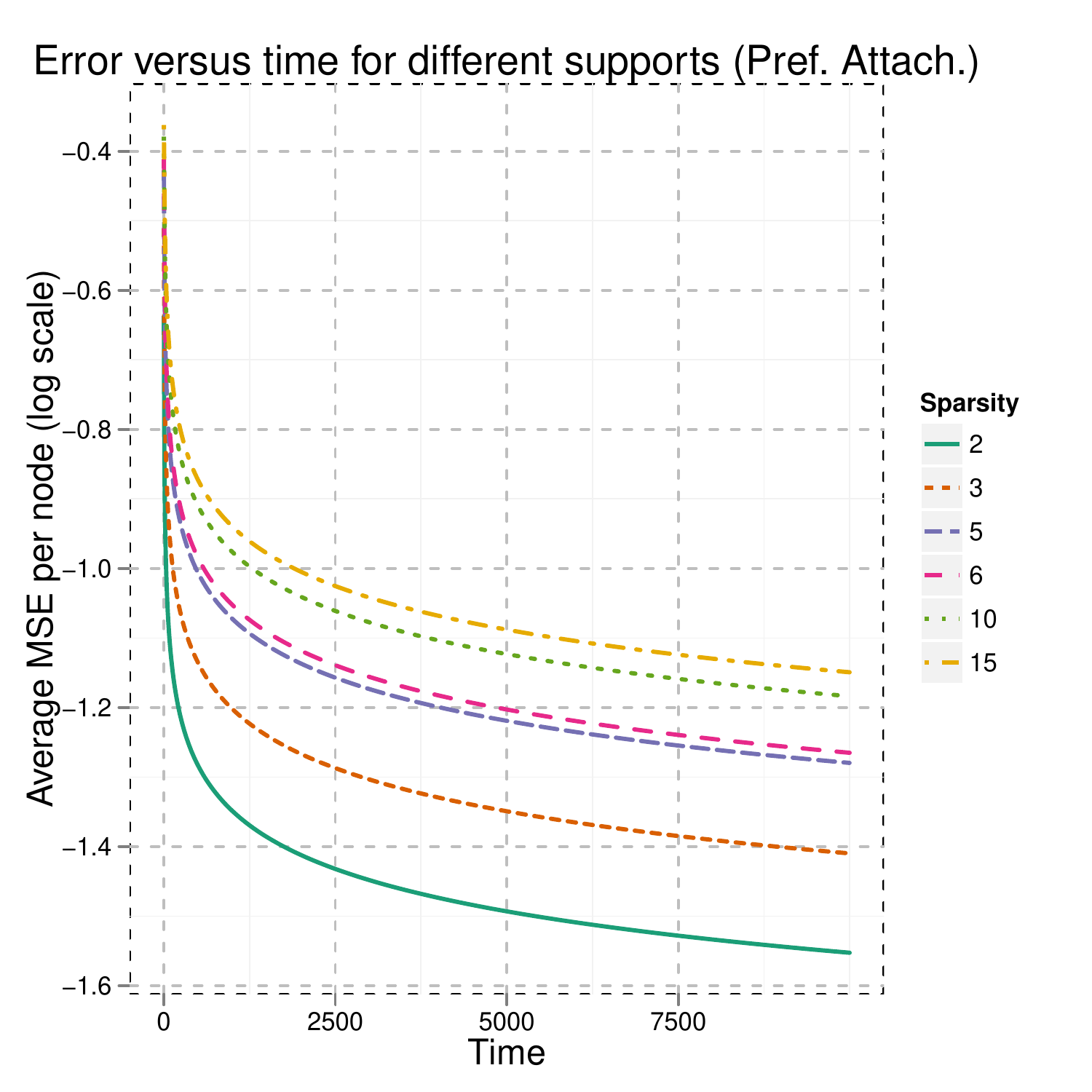}}%
\subfigure[Watts-Strogatz]{\includegraphics[width=2.6in]{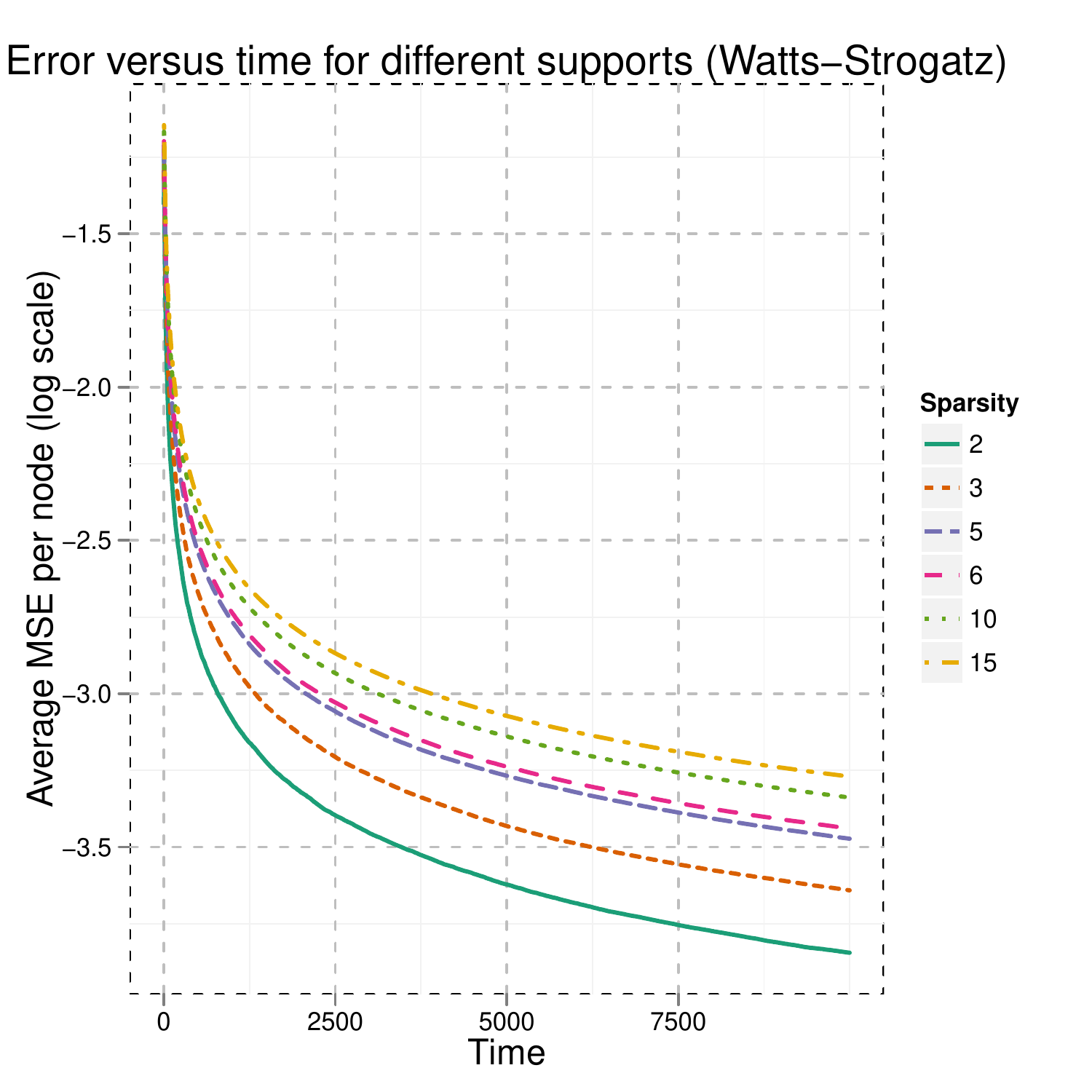}}
\caption{Average MSE per node versus time (on a $\log_{10}(\cdot)$ scale) versus time for different support sizes for four different network topologies.  The distribution is uniform on a subset of size $M^{\ast} = 2$ to $M^{\ast} = 15$ (the Sparsity level)  out of $M = 150$.  The MSE decays much more quickly for sparser distributions for the same alphabet size $M$.  \label{fig:EVTSparse}}
\end{figure*}

\begin{figure*}
\centering
\subfigure[Erd\H{o}s-R\'{e}nyi]{\includegraphics[width=2.6in]{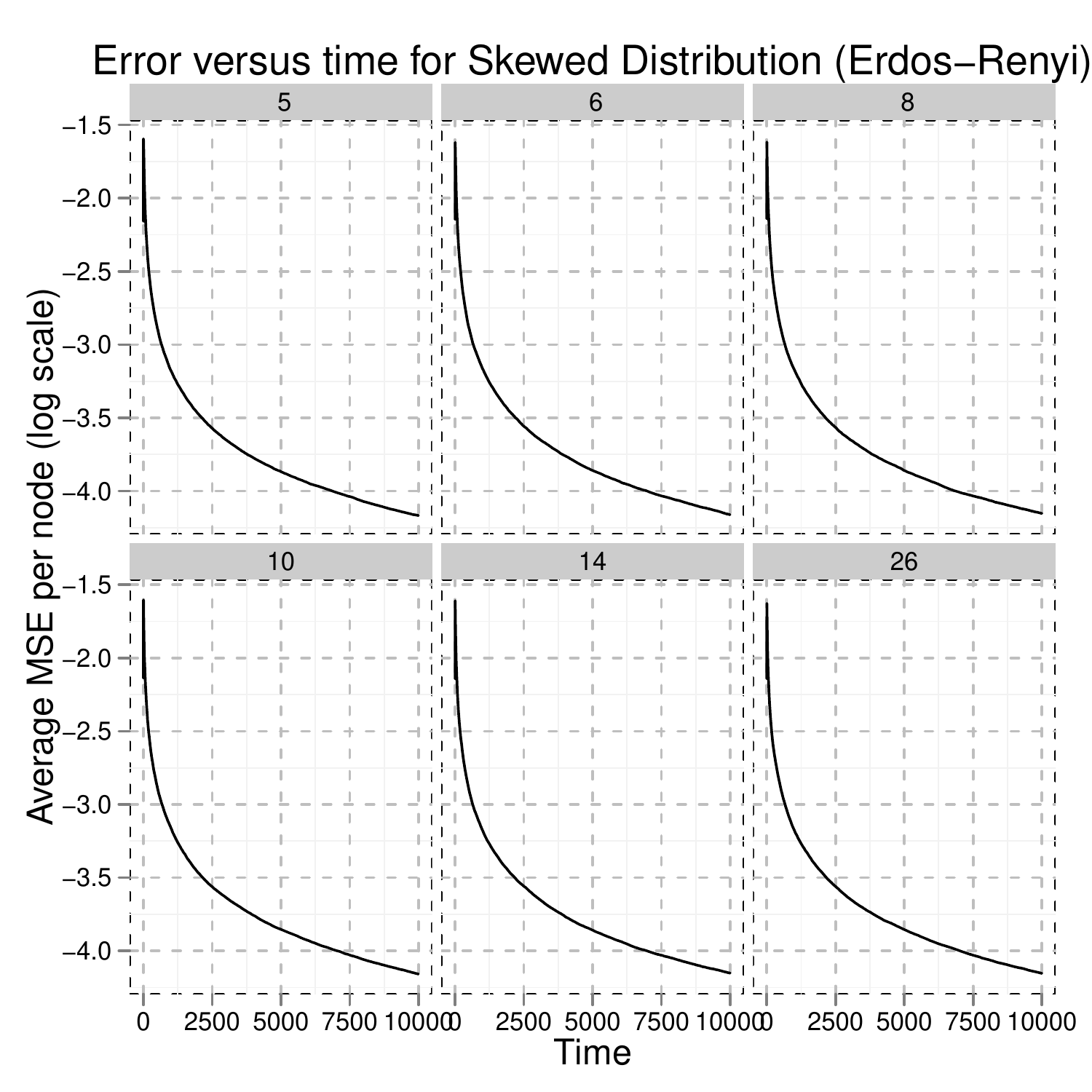}}%
\subfigure[Grid]{\includegraphics[width=2.6in]{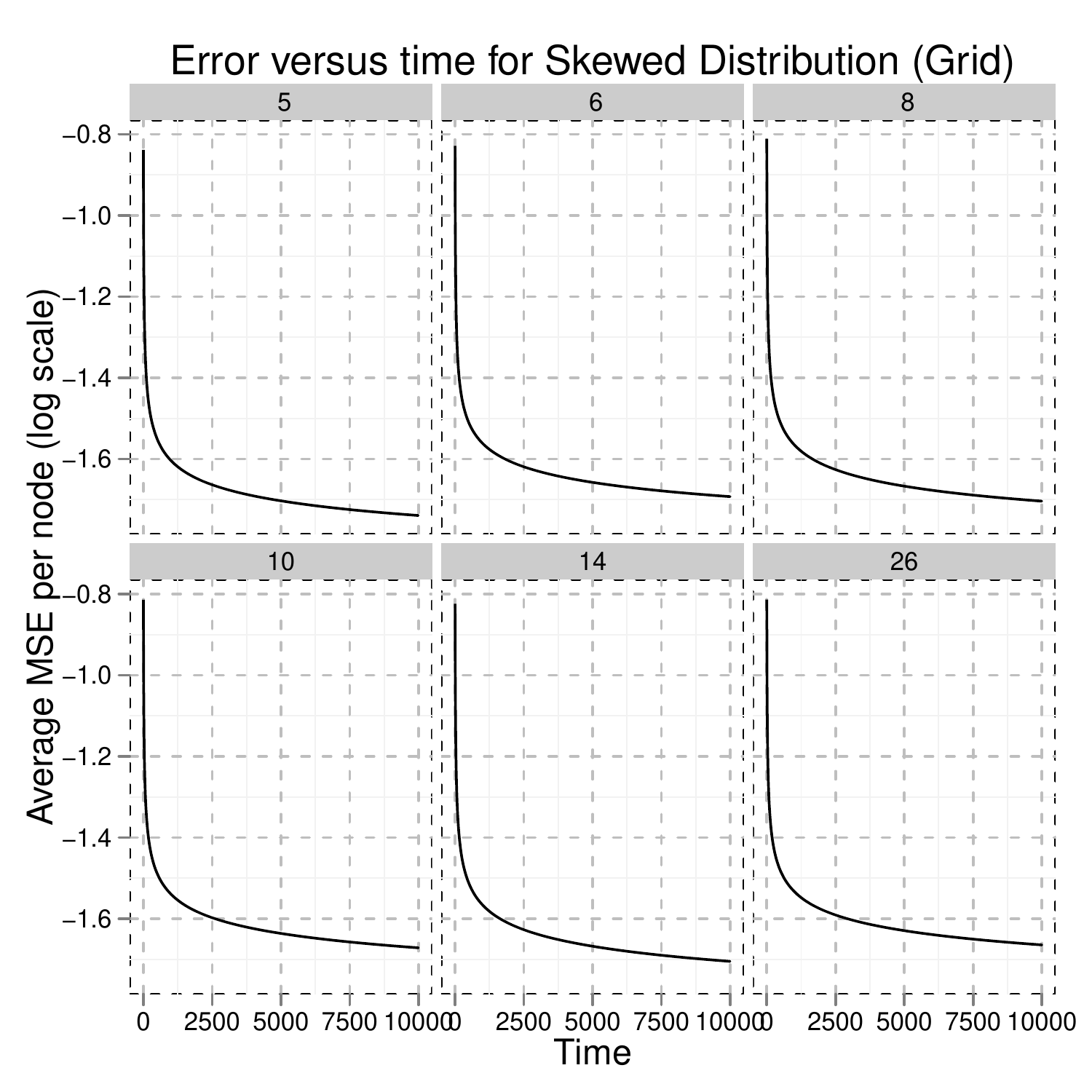}} \\
\subfigure[Preferential attachment]{\includegraphics[width=2.6in]{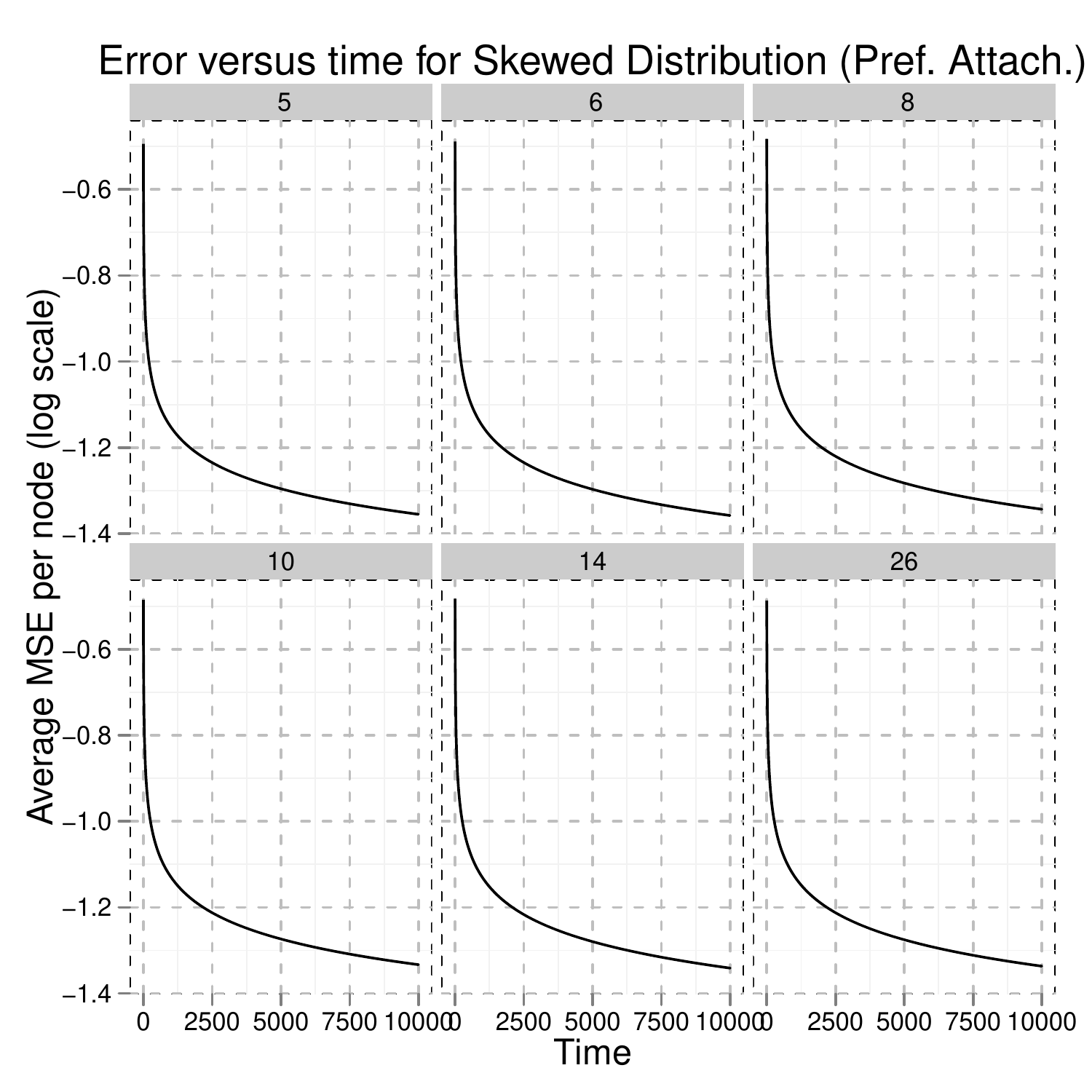}}%
\subfigure[Watts-Strogatz]{\includegraphics[width=2.6in]{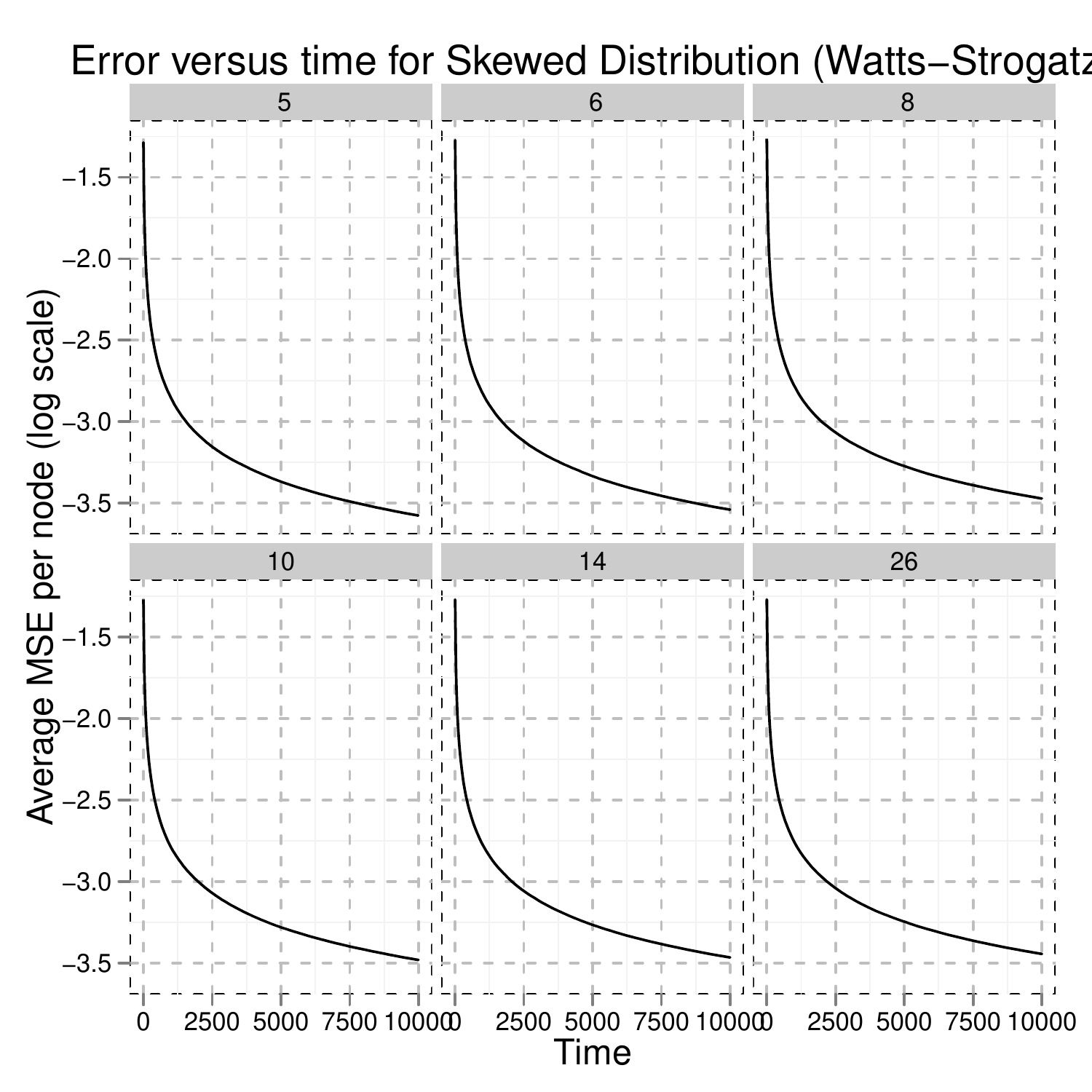}}
\caption{Average MSE per node versus time (on a $\log_{10}(\cdot)$ scale) versus time for different $M$ for a distribution in \eqref{eq:exp:skew} that is skewed with larger $M$.  In general, the error is dominated by the convergence on the larger elements of the histogram.  \label{fig:EVTSkew}}
\end{figure*}

The size and shape of the histogram to be estimated also affects the rate of convergence.  To illustrate this, we sampled  initial values from a uniform distribution on $M$ items for different values of $M$.  Figure \ref{fig:tvm2} shows the average time to get to an MSE of $10^{-2}$ versus $M$ for this scenario.  Here the effect of the network topology is quite pronounced; topological features such as the network diameter seem to have a significant impact on the time to convergence.

To see the effect of the number of nonzero elements in a fixed example we simulated the Erd\H{o}s-R\'{e}nyi, grid, preferential attachment, and Watts-Strogatz models for $n = 100$ with initial values for the nodes sampled from \resp{ a set of sparse distributions with different values of $M$.  
More precisely, we considered a sparse distribution over $M = 150$ bins where the actual distribution of opinions is uniform on an (unknown) sparse subset of $M^{\ast} \ll M$  bins where in our simulations, $M^{\ast}$ ranges from $M^{\ast} = 2$ to $M^{\ast} = 15$. } The effect of this ``sparsity'' is shown in Figure \ref{fig:EVTSparse}, where the log MSE per node is plotted against the number of time steps of the algorithm. Here the difference in the network topologies is more stark -- for the  Erd\H{o}s-R\'{e}nyi graph the effect of changing the number of elements is negligible, but the average MSE per node increases in the other three graph models.  The difference is greatest in the preferential attachment model, where the increase in $M$ corresponds to a nearly linear increase in the log MSE.

Next we consider a closely related question regarding the shape of the histogram to be estimated. In particular, we 
considered initial distributions which are heavily concentrated on a few elements but still contain 
many elements with relatively low popularity.  Specifically, in our simulations we chose the initial values to be drawn from the following 
distribution 
	\begin{align}
	\Pi = \left(0.38, 0.38, \frac{0.24}{M-2}, \ldots, \frac{0.24}{M-2} \right)
	\label{eq:exp:skew}
	\end{align}
for values of $M$ ranging from $M = 5$ to $M = 26$.  We simulated each network 50 times, uniformly assigning the initial values to the nodes.  The average error is shown in Figure \ref{fig:EVTSkew}.  Here we see that when the distribution is biased such that most of the weight is on the first two elements, the support size $M$ does not have an appreciable effect on the convergence time.  What Figure \ref{fig:EVTSkew} suggests is that the shape of the distribution is more important than the support.  This is not surprising -- because we are measuring squared error, elements in $\Pi$ which are small will contribute relatively little to the overall error and so in this sense the uniform distribution is the ``worst case'' for convergence.  In these scenarios, different measures of convergence may be important, such as the Kullback-Leibler divergence between the estimated distributions and $\Pi$.  Other quantities related to $\Pi$ may impact the rate of convergence of the algorithm.

\section{Discussion}

In this paper we studied a simple model of message passing in which nodes or agents communicate random messages generated from their current estimates of a global distribution.  The message model is inspired by models of social messaging in which agents communicate only a part of their current beliefs at each time.  This family of processes contains several interesting instances, including a recent consensus-based model for language formation and an exchange-based algorithm that results in agents learning the true distribution of initial opinions in the network \cite{Narayanan11:lang}.  To analyze this latter process we found a stochastic optimization procedure corresponding to the algorithm.  The simulation results confirm the theory and also show that while the topology of the network affects the rate of convergence, the shape of the overall histogram $\Pi$ may play larger role than its support size when considering $L_2$ convergence.

One interesting theoretical question is whether the error $\sqrt{t} (\mbf{Q}(t) - \mbf{Q}^{\ast})$ converges to a normal distribution when $\delta(t) = 1/t$.  Such a result was obtained by Rajagopal and Wainwright~\cite{5575452} for certain cases of noisy communication in consensus schemes for scalars.  They showed a connection between the network topology and the covariance of the normalized asymptotic error.  Such a result will not transfer immediately to our scenario because of the additional perturbation term $\mbf{C}(t)$.  However, because this term decays rapidly, we do not believe it will impact the covariance matrix.  Characterizing the asymptotic distribution of the error in terms of the graph topology, $M$, and $\Pi$ may yield additional insights into the convergence rates in terms of measures other than $L_2$ norm of error vector.

The results in this paper also apply to the ``gossip'' scenario wherein only one pair of nodes exchanges messages at a time.  This corresponds to selecting a random graph $\mc{G}(t)$ which contains only a single edge.  In terms of time, the convergence in this setting will be slower because only one pair of messages is exchanged in a single time slot.  The analysis framework is fairly general -- to get the almost-sure convergence we need mild assumptions on the message distributions.  Both finding other interesting instances of the algorithm and extending the analysis for metrics such as divergence and other statistical measures are interesting directions for future work.  Solving the latter problem may yield some new techniques for analyzing other statistical procedures which can be cast as stochastic optimization, such as empirical risk minimization.

This model of random message passing may be useful in other contexts such as inference and optimization.  Stochastic coordinate ascent is used in convex optimization over large data sets; extending this framework to the distributed optimization setting is a promising future direction, especially for high-dimensional problems.  In belief propagation, stochastic generation of beliefs can ensure convergence even when the state space is very large~\cite{NoorshamsW:11sbp}.  Finally, the framework here can also be applied to a model for distributed parametric inference in social networks~\cite{MolaviJ:12cdc,JadbabaieMST:12,MolaviRTJ:12acc} in which agents both observe and communicate over time.  In these applications and in others, the same ideas behind the social sampling model in this paper appear to be useful in reducing the message complexity while allowing consistent inference in distributed setting.

\section*{Acknowledgements}

The authors would like to thank Vivek Borkar for helpful suggestions regarding stochastic approximation techniques, and the anonymous reviewers, whose detailed feedback helped improve the organization and exposition of the paper.

\bibliographystyle{IEEEtran}
\bibliography{socialsampling}

\begin{thebibliography}{10}
\providecommand{\url}[1]{#1}
\csname url@samestyle\endcsname
\providecommand{\newblock}{\relax}
\providecommand{\bibinfo}[2]{#2}
\providecommand{\BIBentrySTDinterwordspacing}{\spaceskip=0pt\relax}
\providecommand{\BIBentryALTinterwordstretchfactor}{4}
\providecommand{\BIBentryALTinterwordspacing}{\spaceskip=\fontdimen2\font plus
\BIBentryALTinterwordstretchfactor\fontdimen3\font minus
  \fontdimen4\font\relax}
\providecommand{\BIBforeignlanguage}[2]{{%
\expandafter\ifx\csname l@#1\endcsname\relax
\typeout{** WARNING: IEEEtran.bst: No hyphenation pattern has been}%
\typeout{** loaded for the language `#1'. Using the pattern for}%
\typeout{** the default language instead.}%
\else
\language=\csname l@#1\endcsname
\fi
#2}}
\providecommand{\BIBdecl}{\relax}
\BIBdecl

\bibitem{SarwateJ:12ciss}
\BIBentryALTinterwordspacing
A.~D. Sarwate and T.~Javidi, ``Distributed learning from social sampling,'' in
  \emph{Proceedings of the 46th Annual Conference on Information Sciences and
  Systems (CISS)}, Princeton, NJ, USA, March 2012. [Online]. Available:
  \url{http://dx.doi.org/10.1109/CISS.2012.6310767}
\BIBentrySTDinterwordspacing

\bibitem{SarwateJ:11allerton}
\BIBentryALTinterwordspacing
------, ``Opinion dynamics and distributed learning of distributions,'' in
  \emph{Proceedings of the 49th Annual Allerton Conference on Communication,
  Control and Computation}, Monticello, IL, USA, September 2011. [Online].
  Available: \url{http://dx.doi.org/10.1109/Allerton.2011.6120297}
\BIBentrySTDinterwordspacing

\bibitem{French:56power}
\BIBentryALTinterwordspacing
J.~{French, Jr.}, ``A formal theory of social power,'' \emph{Psychological
  Review}, vol.~63, no.~3, pp. 181--194, May 1956. [Online]. Available:
  \url{http://dx.doi.org/10.1037/h0046123}
\BIBentrySTDinterwordspacing

\bibitem{Harary:59criterion}
F.~Harary, \emph{Studies in Social Power}.\hskip 1em plus 0.5em minus
  0.4em\relax Ann Arbor, MI: Institute for Social Research, 1959, ch. A
  criterion for unanimity in {French's} theory of social power, pp. 168--182.

\bibitem{DeGroot:74consensus}
\BIBentryALTinterwordspacing
M.~H. DeGroot, ``Reaching a consensus,'' \emph{Journal of the American
  Statistical Association}, vol.~69, no. 345, pp. 118--121, 1974. [Online].
  Available: \url{http://www.jstor.org/stable/2285509}
\BIBentrySTDinterwordspacing

\bibitem{Narayanan11:lang}
H.~Narayanan and P.~Niyogi, ``Language evolution, coalescent processes, and the
  consensus problem on a social network,'' October 2011, under revision.

\bibitem{Aumann:76disagree}
\BIBentryALTinterwordspacing
R.~J. Aumann, ``\BIBforeignlanguage{English}{Agreeing to disagree},''
  \emph{\BIBforeignlanguage{English}{The Annals of Statistics}}, vol.~4, no.~6,
  pp. pp. 1236--1239, 1976. [Online]. Available:
  \url{http://www.jstor.org/stable/2958591}
\BIBentrySTDinterwordspacing

\bibitem{BorkarV:82}
\BIBentryALTinterwordspacing
V.~Borkar and P.~Varaiya, ``Asymptotic agreement in distributed estimation,''
  \emph{IEEE Transactions on Automatic Control}, vol. AC-27, no.~3, pp.
  650--655, June 1982. [Online]. Available:
  \url{http://dx.doi.org/10.1109/TAC.1982.1102982}
\BIBentrySTDinterwordspacing

\bibitem{TsitsiklisA:84}
\BIBentryALTinterwordspacing
J.~N. Tsitsiklis and M.~Athans, ``Convergence and asymptotic agreement in
  distributed decision problems,'' \emph{IEEE Transactions on Automatic
  Control}, vol. AC-29, no.~1, pp. 42--50, January 1984. [Online]. Available:
  \url{http://dx.doi.org/10.1109/TAC.1984.1103385}
\BIBentrySTDinterwordspacing

\bibitem{Fax:01thesis}
J.~Fax, ``Optimal and cooperative control of vehicle formation,'' Ph.D.
  dissertation, California Institute of Technolology, Pasadena, CA, 2001.

\bibitem{FaxM:04flow}
\BIBentryALTinterwordspacing
J.~Fax and R.~Murray, ``Information flow and cooperative control of vehicle
  formations,'' \emph{IEEE Transactions on Automatic Control}, vol.~49, no.~9,
  pp. 1465--1476, September 2004. [Online]. Available:
  \url{http://dx.doi.org/10.1109/TAC.2004.834433}
\BIBentrySTDinterwordspacing

\bibitem{OlfatiSaberM:04consensus}
\BIBentryALTinterwordspacing
R.~Olfati-Saber and R.~Murray, ``Consensus problems in networks of agents with
  switching topology and time-delays,'' \emph{IEEE Transactions on Automatic
  Control}, vol.~49, no.~9, pp. 1520--1533, September 2004. [Online].
  Available: \url{http://dx.doi.org/10.1109/TAC.2004.834113}
\BIBentrySTDinterwordspacing

\bibitem{AgaevC:00digraph}
\BIBentryALTinterwordspacing
R.~Agaev and P.~Chebotarev, ``The matrix of maximum out forests of a digraph
  and its applications,'' \emph{Automation and Remote Control}, vol.~61, pp.
  1424--1450, September 2000. [Online]. Available:
  \url{http://dx.doi.org/10.1023/A:1002862312617}
\BIBentrySTDinterwordspacing

\bibitem{AgaevC:01spanning}
\BIBentryALTinterwordspacing
R.~Agaev and P.~Y. Chebotarev, ``Spanning forests of a digraph and their
  applications,'' \emph{Automation and Remote Control}, vol.~62, pp. 443--466,
  March 2001. [Online]. Available:
  \url{http://dx.doi.org/10.1023/A:1002862312617}
\BIBentrySTDinterwordspacing

\bibitem{ChebotarevA:02forest}
\BIBentryALTinterwordspacing
P.~Chebotarev and R.~Agaev, ``Forest matrices around the {Laplacian} matrix,''
  \emph{Linear Algebra and its Applications}, vol. 356, pp. 253--274, 2002.
  [Online]. Available: \url{http://dx.doi.org/10.1016/S0024-3795(02)00388-9}
\BIBentrySTDinterwordspacing

\bibitem{5485031}
\BIBentryALTinterwordspacing
P.~Chebotarev, ``{Comments on ``Consensus and Cooperation in Networked
  Multi-Agent Systems''},'' \emph{Proceedings of the IEEE}, vol.~98, no.~7, pp.
  1353 --1354, July 2010. [Online]. Available:
  \url{http://dx.doi.org/10.1109/JPROC.2010.2049911}
\BIBentrySTDinterwordspacing

\bibitem{5485032}
\BIBentryALTinterwordspacing
R.~Olfati-Saber, J.~Fax, and R.~Murray, ``{Reply to ``Comments on `Consensus
  and Cooperation in Networked Multi-Agent Systems' ''},'' \emph{Proceedings of
  the IEEE}, vol.~98, no.~7, pp. 1354 --1355, July 2010. [Online]. Available:
  \url{http://dx.doi.org/10.1109/JPROC.2010.2049912}
\BIBentrySTDinterwordspacing

\bibitem{BoydIT}
\BIBentryALTinterwordspacing
S.~Boyd, A.~Ghosh, B.~Prabhakar, and D.~Shah, ``Randomized gossip algorithms,''
  \emph{IEEE Transactions on Information Theory}, vol.~52, no.~6, pp.
  2508--2530, June 2006. [Online]. Available:
  \url{http://dx.doi.org/10.1109/TIT.2006.874516}
\BIBentrySTDinterwordspacing

\bibitem{MurraySurvey}
\BIBentryALTinterwordspacing
R.~Olfati-Saber, J.~Fax, and R.~M. Murray, ``Consensus and cooperation in
  networked multi-agent systems,'' \emph{Proceedings of the IEEE}, vol.~95,
  no.~1, pp. 215--233, January 2007. [Online]. Available:
  \url{http://dx.doi.org/10.1109/JPROC.2006.887293}
\BIBentrySTDinterwordspacing

\bibitem{Fagn08}
\BIBentryALTinterwordspacing
F.~Fagnani and S.~Zampieri, ``Randomized consensus algorithms over large scale
  networks,'' \emph{IEEE Journal on Selected Areas in Communication}, vol.~26,
  no.~4, pp. 634--649, May 2008. [Online]. Available:
  \url{http://dx.doi.org/10.1109/JSAC.2008.080506}
\BIBentrySTDinterwordspacing

\bibitem{journals/siamco/OlshevskyT09}
\BIBentryALTinterwordspacing
A.~Olshevsky and J.~N. Tsitsiklis, ``Convergence speed in distributed consensus
  and averaging.'' \emph{SIAM J. Control and Optimization}, vol.~48, no.~1, pp.
  33--55, 2009. [Online]. Available: \url{http://dx.doi.org/10.1137/060678324}
\BIBentrySTDinterwordspacing

\bibitem{DimakisSW:08gossip}
\BIBentryALTinterwordspacing
A.~Dimakis, A.~Sarwate, and M.~Wainwright, ``Geographic gossip: Efficient
  averaging for sensor networks,'' \emph{IEEE Transactions on Signal
  Processing}, vol.~56, no.~3, pp. 1205--1216, 2008. [Online]. Available:
  \url{http://dx.doi.org/10.1109/TSP.2007.908946}
\BIBentrySTDinterwordspacing

\bibitem{BDTVPath}
\BIBentryALTinterwordspacing
F.~Benezit, A.~G. Dimakis, P.~Thiran, and M.~Vetterli, ``Order-optimal
  consensus through randomized path averaging,'' \emph{IEEE Transactions on
  Information Theory}, vol.~56, no.~10, pp. 5150--5167, October 2010. [Online].
  Available: \url{http://dx.doi.org/10.1109/TIT.2010.2060050}
\BIBentrySTDinterwordspacing

\bibitem{Tunc09}
\BIBentryALTinterwordspacing
T.~C. Aysal, M.~E. Yildiz, A.~D. Sarwate, and A.~Scaglione, ``Broadcast gossip
  algorithms for consensus,'' \emph{IEEE Transactions on Signal Processing},
  vol.~57, no.~7, pp. 2748--2761, Jul. 2009. [Online]. Available:
  \url{http://dx.doi.org/10.1109/TSP.2009.2016247}
\BIBentrySTDinterwordspacing

\bibitem{SarwateD:12mobility}
\BIBentryALTinterwordspacing
A.~D. Sarwate and A.~G. Dimakis, ``The impact of mobility on gossip
  algorithms,'' \emph{IEEE Transactions on Information Theory}, vol.~58, no.~3,
  pp. 1731--1742, March 2012. [Online]. Available:
  \url{http://dx.doi.org/10.1109/TIT.2011.2177753}
\BIBentrySTDinterwordspacing

\bibitem{Dimakis10gossipsurvey}
\BIBentryALTinterwordspacing
A.~G. Dimakis, S.~Kar, J.~M. Moura, M.~G. Rabbat, and A.~Scaglione, ``Gossip
  algorithms for distributed signal processing,'' \emph{Proceedings of the
  IEEE}, vol.~98, no.~11, pp. 1847--1864, November 2010. [Online]. Available:
  \url{http://dx.doi.org/10.1109/JPROC.2010.2052531}
\BIBentrySTDinterwordspacing

\bibitem{Aysal07}
\BIBentryALTinterwordspacing
T.~C. Aysal, M.~J. Coates, and M.~G. Rabbat, ``Distributed average consensus
  with dithered quantization,'' \emph{IEEE Transactions on Signal Processing},
  vol.~56, no.~10, pp. 4905--4918, 2008. [Online]. Available:
  \url{http://dx.doi.org/10.1109/TSP.2008.927071}
\BIBentrySTDinterwordspacing

\bibitem{NedicOOT:09dist}
\BIBentryALTinterwordspacing
A.~Nedic, A.~Olshevsky, A.~Ozdaglar, and J.~Tsitsiklis, ``On distributed
  averaging algorithms and quantization effects,'' \emph{IEEE Transactions on
  Automatic Control}, vol.~54, no.~11, pp. 2506--2517, November 2009. [Online].
  Available: \url{http://dx.doi.org/10.1109/TAC.2009.2031203}
\BIBentrySTDinterwordspacing

\bibitem{CarliBZ:10dynamic}
\BIBentryALTinterwordspacing
R.~Carli, F.~Bullo, and S.~Zampieri, ``Quantized average consensus via dynamic
  coding/decoding schemes,'' \emph{International Journal of Robust and
  Nonlinear Control}, vol.~20, no.~2, pp. 156--175, 2010. [Online]. Available:
  \url{http://dx.doi.org/10.1002/rnc.1463}
\BIBentrySTDinterwordspacing

\bibitem{KashyapBS:07quant}
\BIBentryALTinterwordspacing
A.~Kashyap, T.~Ba\c{s}ar, and R.~Srikant, ``Quantized consensus,''
  \emph{Automatica}, vol.~43, no.~7, pp. 1192--1203, 2007. [Online]. Available:
  \url{http://dx.doi.org/10.1016/j.automatica.2007.01.002}
\BIBentrySTDinterwordspacing

\bibitem{CarliFFZ:10quantgossip}
\BIBentryALTinterwordspacing
R.~Carli, F.~Fagnani, P.~Frasca, and S.~Zampieri, ``Gossip consensus algorithms
  via quantized communication,'' \emph{Automatica}, vol.~46, no.~1, pp. 70--80,
  January 2010. [Online]. Available:
  \url{http://dx.doi.org/10.1016/j.automatica.2009.10.032}
\BIBentrySTDinterwordspacing

\bibitem{ZhuM:11time}
\BIBentryALTinterwordspacing
M.~Zhu and S.~Mart\'{i}nez, ``On the convergence time of asynchronous
  distributed quantized averaging algorithms,'' \emph{IEEE Transactions on
  Automatic Control}, vol.~56, no.~2, pp. 386--390, February 2011. [Online].
  Available: \url{http://dx.doi.org/10.1109/TAC.2010.2093276}
\BIBentrySTDinterwordspacing

\bibitem{LavaeiM:10gossip}
\BIBentryALTinterwordspacing
J.~Lavaei and R.~M. Murray, ``Quantized consensus by means of gossip
  algorithm,'' \emph{IEEE Transactions on Automatic Control}, vol.~57, no.~1,
  pp. 19--32, January 2012. [Online]. Available:
  \url{http://dx.doi.org/10.1109/TAC.2011.2160593}
\BIBentrySTDinterwordspacing

\bibitem{5719290}
\BIBentryALTinterwordspacing
K.~Srivastava and A.~Nedic, ``Distributed asynchronous constrained stochastic
  optimization,'' \emph{IEEE Journal of Selected Topics in Signal Processing},
  vol.~5, no.~4, pp. 772 --790, August 2011. [Online]. Available:
  \url{http://dx.doi.org/10.1109/JSTSP.2011.2118740}
\BIBentrySTDinterwordspacing

\bibitem{YildizS:08coding}
\BIBentryALTinterwordspacing
M.~Yildiz and A.~Scaglione, ``Coding with side information for rate-constrained
  consensus,'' \emph{IEEE Transactions on Signal Processing}, vol.~56, no.~8,
  pp. 3753 --3764, August 2008. [Online]. Available:
  \url{http://dx.doi.org/10.1109/TSP.2008.919636}
\BIBentrySTDinterwordspacing

\bibitem{KarM:10quant}
\BIBentryALTinterwordspacing
S.~Kar and J.~Moura, ``Distributed consensus algorithms in sensor networks:
  Quantized data and random link failures,'' \emph{IEEE Transactions on Signal
  Processing}, vol.~58, no.~3, pp. 1383 --1400, March 2010. [Online].
  Available: \url{http://dx.doi.org/10.1109/TSP.2009.2036046}
\BIBentrySTDinterwordspacing

\bibitem{5575452}
\BIBentryALTinterwordspacing
R.~Rajagopal and M.~Wainwright, ``Network-based consensus averaging with
  general noisy channels,'' \emph{IEEE Transactions on Signal Processing},
  vol.~59, no.~1, pp. 373--385, January 2011. [Online]. Available:
  \url{http://dx.doi.org/10.1109/TSP.2010.2077282}
\BIBentrySTDinterwordspacing

\bibitem{KushnerYin:2010}
H.~J. Kushner and G.~G. Yin, \emph{Stochastic Approximation and Recursive
  Algorithms and Applications}, 2nd~ed.\hskip 1em plus 0.5em minus 0.4em\relax
  Springer, 2010.

\bibitem{BMP:adaptive}
A.~Benveniste, M.~M\'{e}tivier, and P.~Priouret, \emph{Adaptive Algorithms and
  Stochastic Approximations}, ser. Applications of Mathematics.\hskip 1em plus
  0.5em minus 0.4em\relax Berlin: Springer-Verlag, 1990, no.~22.

\bibitem{Erdos:1960}
\BIBentryALTinterwordspacing
P.~Erd\H{o}s and A.~R\'{e}nyi, ``On the evolution of random graphs,'' in
  \emph{Publications of the Mathematical Institute of the Hungarian Academy of
  Sciences}, vol.~5, 1960, pp. 17--61. [Online]. Available:
  \url{https://www.renyi.hu/~p_erdos/1960-10.pdf}
\BIBentrySTDinterwordspacing

\bibitem{barabasi1999emergence}
\BIBentryALTinterwordspacing
A.-L. Barab{\'a}si and R.~Albert, ``{Emergence of scaling in random
  networks},'' \emph{Science}, vol. 286, no. 5439, p. 509, October 1999.
  [Online]. Available: \url{http://dx.doi.org/10.1126/science.286.5439.509}
\BIBentrySTDinterwordspacing

\bibitem{BollobasR04}
\BIBentryALTinterwordspacing
B.~Bollob{\'a}s and O.~Riordan, ``The diameter of a scale-free random graph,''
  \emph{Combinatorica}, vol.~24, no.~1, pp. 5--34, 2004. [Online]. Available:
  \url{http://dx.doi.org/10.1007/s00493-004-0002-2}
\BIBentrySTDinterwordspacing

\bibitem{watts1998collective}
D.~J. Watts and S.~H. Strogatz, ``{Collective dynamics of `small-world'
  networks},'' \emph{Nature}, vol. 393, no. 6684, pp. 440--442, 1998.

\bibitem{igraph}
G.~Cs\'{a}rdi and T.~Nepusz, ``The igraph software package for complex network
  research,'' \emph{InterJournal Complex Systems}, no. 1695, 2006.

\bibitem{NoorshamsW:11sbp}
\BIBentryALTinterwordspacing
N.~Noorshams and M.~J. Wainwright, ``Stochastic belief propagation: A
  low-complexity alternative to the sum-product algorithm,'' \emph{IEEE
  Transactions on Information Theory}, vol.~59, no.~4, pp. 1981--2000, April
  2013. [Online]. Available: \url{http://dx.doi.org/10.1109/TIT.2012.2231464}
\BIBentrySTDinterwordspacing

\bibitem{MolaviJ:12cdc}
\BIBentryALTinterwordspacing
P.~Molavi and A.~Jadbabaie, ``Network structure and efficiency of observational
  social learning,'' in \emph{Proceedings of the 51st IEEE Conference on
  Decision and Control}, Maui, HI, USA, 2012. [Online]. Available:
  \url{http://dx.doi.org/10.1109/CDC.2012.6426454}
\BIBentrySTDinterwordspacing

\bibitem{JadbabaieMST:12}
\BIBentryALTinterwordspacing
A.~Jadbabaie, P.~Molavi, A.~Sandroni, and A.~Tahbaz-Salehi, ``Non-{B}ayesian
  social learning,'' \emph{Games and Economic Behavior}, vol.~76, no.~1, pp.
  210--225, 2012. [Online]. Available:
  \url{http://dx.doi.org/10.1016/j.geb.2012.06.001}
\BIBentrySTDinterwordspacing

\bibitem{MolaviRTJ:12acc}
P.~Molavi, K.~R. Rad, A.~Tahbaz-Salehi, and A.~Jadbabaie, ``On consensus and
  exponentially fast social learning,'' in \emph{Proceedings of the 2012
  American Control Conference}, 2012.

\end{thebibliography}

\end{document}